\documentclass[12pt]{amsart}

\usepackage{graphicx,pdfpages,lscape,longtable,adjustbox,tikz,array,geometry,subcaption}
\usepackage[fontsize=14pt]{scrextend}
\usepackage{amsthm}
\usepackage{enumitem}
\usepackage[]{url}
\makeatletter
\renewcommand\@biblabel[1]{}
\makeatother

\newtheorem{theorem}{Theorem}
\newtheorem{prop}[theorem]{Proposition}
\newtheorem{lemma}[theorem]{Lemma}
\newtheorem{cor}[theorem]{Corollary}
\newtheorem{conj}[theorem]{Conjecture}
\theoremstyle{definition}
\newtheorem{definition}[theorem]{Definition}

\theoremstyle{remark}
\newtheorem{remark}[theorem]{Remark}

\newcolumntype{R}[2]{%
    >{\adjustbox{angle=#1,lap=\width-(#2)}\bgroup}%
    l%
    <{\egroup}%
	}

\numberwithin{theorem}{section}
\newcommand{\R}{\mathbb{R}}

\newcommand{\paren}[1]{\left(#1\right)}
\newcommand{\set}[1]{\left\{#1\right\}}
\newcommand{\bracket}[1]{\left[#1\right]}
\newcommand{\abs}[1]{\left\lvert#1\right\rvert}
\newcommand{\mb}{\mathbb}

\newcommand\numcasesinitial{2074}
\newcommand\numcasesaftereintmin{29}
\newcommand\numcasesafterelim{7}

\newcommand{\tetrahedron}[7]{
\begin{tikzpicture}[scale=#7,transform shape]
\draw [black] (0,1.5) node [anchor=south] {1} -- (-1,0) node [anchor=east] {2};
\node [anchor=south east] at (-.4,.6) {$#1$};
\draw [black] (0,1.5) -- (0,-.5) node [anchor=north] {3};
\node [anchor=east] at (.05,.6) {$#2$};
\draw [black] (0,1.5) -- (1,0) node [anchor=west] {4};
\node [anchor=south west] at (.4,.6) {$#3$};
\draw [black] (-1,0) -- (0,-.5);
\node [anchor=north east] at (-.4,-.2) {$#4$};
\draw [dashed, black] (-1,0) -- (1,0);
\node [anchor=south] at (.4,-.1) {$#5$};
\draw [black] (0,-.5) -- (1,0);
\node [anchor=north west] at (.4,-.2) {$#6$};
\end{tikzpicture}
}

\newcommand{\pathlentwo}[5]{
\begin{tikzpicture}[scale=#5, transform shape]
\draw [black] (0,0) node [anchor=east] {$#1$}
-- node [anchor=south] {#3}
(#4,0) node[anchor=west] {$#2$};
\end{tikzpicture}
}

\newcommand{\pathlenthree}[8]{
\begin{tikzpicture}[scale=#8, transform shape]
\node (0,0) {$#2$};
\draw [black] (-#6/2,0)
-- node [anchor=south] {#4}
(-#7-#6/2,0) node [anchor=east] {$#1$};
\draw [black] (#6/2,0)
-- node [anchor=south] {#5}
(#7+#6/2,0) node [anchor=west] {$#3$};
\end{tikzpicture}
}

\newcommand{\pathlenfour}[9]{
\def\ArgI{#1}
\def\ArgII{#2}
\def\ArgIII{#3}
\def\ArgIV{#4}
\def\ArgV{#5}
\def\ArgVI{#6}
\def\ArgVII{#7}
\def\ArgVIII{#8}
\def\ArgIX{#9}
\pathlenfourcontinued
}

\newcommand{\pathlenfourcontinued}[3]{
\begin{tikzpicture}[scale=#3, transform shape]
\draw [black] (-\ArgVIII/2,0) node [anchor=east] {$\ArgII$}
-- node [anchor=south] {\ArgVI}
(\ArgVIII/2,0) node [anchor=west] {$\ArgIII$};
\draw [black] (-\ArgVIII/2-\ArgIX,-#1)
-- node [anchor=east] {\ArgV}
(-\ArgVIII/2-\ArgIX,-#1-#2) node [anchor=north] {$\ArgI$};
\draw [black] (\ArgVIII/2+\ArgIX,-#1)
-- node [anchor=west] {\ArgVII}
(\ArgVIII/2+\ArgIX,-#1-#2) node [anchor=north] {$\ArgIV$};
\end{tikzpicture}
}

\newcommand{\pathlenfive}[9]{
\def\ArgI{#1}
\def\ArgII{#2}
\def\ArgIII{#3}
\def\ArgIV{#4}
\def\ArgV{#5}
\def\ArgVI{#6}
\def\ArgVII{#7}
\def\ArgVIII{#8}
\def\ArgIX{#9}
\pathlenfivecontinued
}

\newcommand{\pathlenfivecontinued}[6]{
\begin{tikzpicture}[scale=#6, transform shape]
\node (0,0) {$\ArgIII$};
\draw [black] (-#1/2,0)
-- node [anchor=south] {\ArgVII}
(-#1/2-#2,0) node [anchor=east] {$\ArgII$};
\draw [black] (-#1/2-#2-#3,-#4)
-- node [anchor=east] {\ArgVI}
(-#1/2-#2-#3,-#4-#5) node [anchor=north] {$\ArgI$};
\draw [black] (#1/2,0)
-- node [anchor=south] {\ArgVIII}
(#1/2+#2,0) node [anchor=west] {$\ArgIV$};
\draw [black] (#1/2+#2+#3,-#4)
-- node [anchor=west] {\ArgIX}
(#1/2+#2+#3,-#4-#5) node [anchor=north] {$\ArgV$};
\end{tikzpicture}
}

\newcommand{\trianglegraph}[9]{
\def\ArgI{#1}
\def\ArgII{#2}
\def\ArgIII{#3}
\def\ArgIV{#4}
\def\ArgV{#5}
\def\ArgVI{#6}
\def\ArgVII{#7}
\def\ArgVIII{#8}
\def\ArgIX{#9}
\trianglegraphcontinued
}

\newcommand{\trianglegraphcontinued}[4]{
\begin{tikzpicture}[scale=#4, transform shape]
\node (0,0) {$\ArgII$};
\draw [black] (-\ArgVII/2,-\ArgVIII)
-- (-\ArgVII/2-\ArgIX,-\ArgVIII-#1);
\node [anchor=east] at (-\ArgVII/2-\ArgIX/2,-\ArgVIII-#1/2+#3) {\ArgIV};
\draw [black] (\ArgVII/2,-\ArgVIII)
-- (\ArgVII/2+\ArgIX,-\ArgVIII-#1);
\node [anchor=west] at (\ArgVII/2+\ArgIX/2,-\ArgVIII-#1/2+#3) {\ArgV};
\draw [black] (-\ArgVII/2-\ArgIX,-\ArgVIII-#1-#2) node [anchor=east] {$\ArgI$}
-- node [anchor=north] {\ArgVI}
(\ArgVII/2+\ArgIX,-\ArgVIII-#1-#2) node [anchor=west] {$\ArgIII$};
\end{tikzpicture}
}

\begin{document}

\title{The Least-Area Tetrahedral Tile of Space}

\author{Eliot Bongiovanni, Alejandro Diaz, Arjun Kakkar, Nat Sothanaphan}

\begin{abstract}
We determine the least-area unit-volume tetrahedral tile of Euclidean space, without the constraint of Gallagher \emph{et al.} that the tiling uses only orientation-preserving images of the tile. The winner remains Sommerville's type 4v.
\end{abstract}

\maketitle

\setcounter{tocdepth}{1}
\tableofcontents

\section{Introduction}
Gallagher \emph{et al.} \cite{Gal} showed that the least-surface-area, unit-volume, face-to-face, tetrahedral tile is Sommerville's \cite{So} type 4v, referred to in our paper as Sommerville No. 1 and illustrated in Figure \ref{fig:sommerville1},
assuming that the tiling is orientation preserving (does not need to reflect the tile).

\begin{figure}
\centering
\begin{subfigure}{.5\textwidth}
\begin{tikzpicture}[scale=0.5,transform shape]
\draw [black] (4,-1.2) -- (0,-12.5);
\draw [black] (-3.8,-9.5) -- (0,-12.5);
\draw [black] (0,-12.5) -- (7.5,-9.5);
\draw [dashed,black] (-3.8,-9.5) -- (7.5,-9.5);
\draw [black] (4,-1.2) -- (-3.8,-9.5);
\draw [black] (4,-1.2) -- (7.5,-9.5);

\node [anchor=center] at (-1.5, -5.5){\LARGE $2, \pi/2$};
\node [anchor=north west] at (3.8, -11){\LARGE $2, \pi/2$};
\node [anchor=center] at (4, -9){\LARGE $\sqrt{3}, \pi/3$};
\node [anchor=center] at (7, -4.8){\LARGE $\sqrt{3}, \pi/3$};
\node [anchor=center] at (-3, -12){\LARGE $\sqrt{3}, \pi/3$};
\node [anchor=center] at (2, -7){\LARGE $\sqrt{3}, \pi/3$};
\end{tikzpicture}
   \centering
   \caption{}
   \label{fig:som1tile_sub1}
\end{subfigure}
\begin{subfigure}{.4\textwidth}
\begin{tikzpicture}[scale=0.5,transform shape]
\draw [blue] (0,0) -- (-3.8,-9.5);
\draw [blue] (10.5,0) -- (7.5,-9.5);
\draw [red] (4,-1.2) -- (0,-12.5);

\draw [blue] (0,0) -- (10.5,0);
\draw [blue] (0,0) -- (4,-1.2);
\draw [blue] (4,-1.2) -- (10.5,0);

\draw [dashed, red] (-3.8,-9.5) -- (7.5,-9.5);
\draw [red] (-3.8,-9.5) -- (0,-12.5);
\draw [red] (0,-12.5) -- (7.5,-9.5);

\draw [red] (4,-1.2) -- (-3.8,-9.5);
\draw [red] (4,-1.2) -- (7.5,-9.5);
\draw [dashed,blue] (0,0) -- (7.5,-9.5);

\node [anchor=center] at (3.0, -11.5){\LARGE $2$};
\node [anchor=center] at (4, -9.5){\LARGE $\sqrt{3}$};
\node [anchor=center] at (-2, -11.5){\LARGE $\sqrt{3}$};
\node [anchor=center] at (2, -7){\LARGE $\sqrt{3}$};

\node [anchor=center] at (6.5, -1){\LARGE $2$};
\node [anchor=center] at (2, -1){\LARGE $\sqrt{3}$};
\node [anchor=center] at (5, 0.3){\LARGE $\sqrt{3}$};

\node [anchor=center] at (-1.9, -4.8){\LARGE $\sqrt{3}$};
\node [anchor=center] at (9, -4.8){\LARGE $\sqrt{3}$};
\end{tikzpicture}
  \centering
  \caption{}
  \label{fig:som1tile_sub2}
\end{subfigure}
\caption{(\ref{fig:som1tile_sub1}): The Sommerville No. 1 tetrahedral tile labeled with edge lengths and dihedral angles. (\ref{fig:som1tile_sub2}): The arrangement of three copies of the Sommerville No. 1 tile into
a triangular prism, proving that it tiles space.}
\label{fig:sommerville1}
\end{figure}
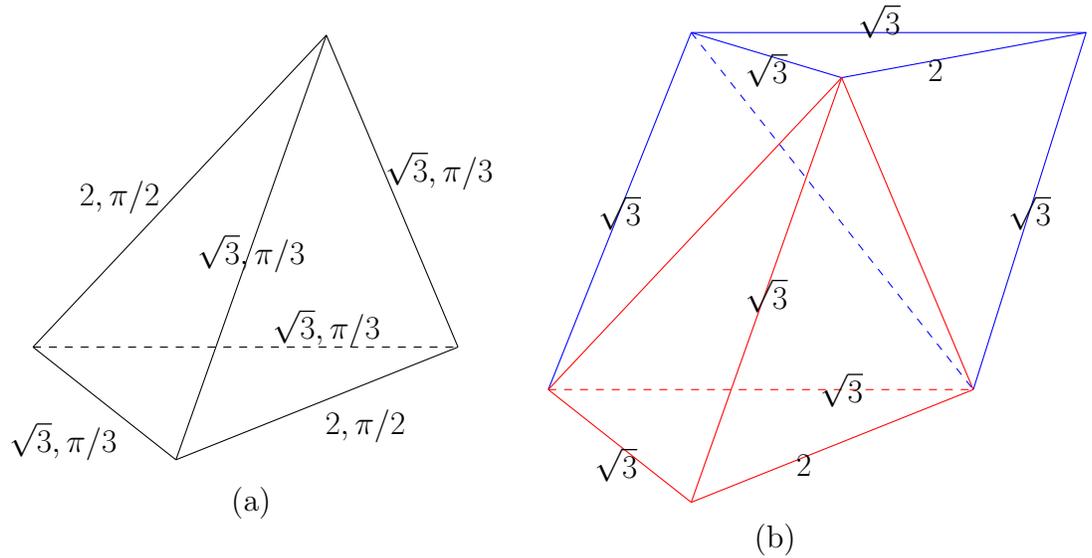

We prove that the Sommerville No. 1 even among non-orientation-preserving tilings is still surface area minimizing.

\begin{theorem}[Thm. \ref{thm:bigtheorem}]
The least-area, face-to-face tetrahedral tile
of unit volume is uniquely the Sommerville No. 1.
\end{theorem}

The least-surface-area property in question could be interpreted as being the closest one to the regular tetrahedron among all face-to-face space-tiling tetrahedra.

Currently, there is no list of all tetrahedral tiles.
Moreover, there are infinitely many tetrahedra that tile in a non-orientation-preserving manner, all known examples provided by Goldberg \cite{Gol}.
We characterize a subclass of tetrahedra that includes all possible candidates for surface area minimization and find that there are none that beat the Sommerville No. 1. First, we classify all tetrahedra into 25 types (Sect. \ref{sec:classification}), fifteen with the tiles completely identified. The Sommerville No. 1 is the least-area tile among these fifteen types. 
The remaining ten types can be reduced to a total of \numcasesinitial\ cases by certain linear equations and bounds on the dihedral angles (Sects. \ref{sec:10types}--\ref{sec:code}). 
In particular, every dihedral angle must be greater than 36.5 degrees (Cor. \ref{cor:dihedralbdsommerville}).
Necessary conditions
for dihedral angles of tetrahedra of the given type, implemented by careful and rigorous computer code, reduce the cases from \numcasesinitial\ to seven (Sect. \ref{sec:code}). These seven cases are eliminated by special arguments (Sect. \ref{sec:remainingcases}), leaving no candidate to tile with less surface area than the Sommerville No. 1.

Trivial Dehn invariant is a necessary but not sufficient condition for a polyhedron to tile and is therefore a potentially useful tool in narrowing down to potential tiles. Unfortunately, we were not able to find a useful way to compute Dehn invariants.

As a consequence of searching for the surface-area-minimizing tetrahedral tile without a complete list of tetrahedral tiles, this paper also completely characterizes the tiles of fifteen of 25 types and provides some necessary conditions for the remaining ten types to tile. This paper is only concerned with surface area minimization, but those interested in tetrahedral tiles in general may find this information useful. The known properties of all 25 types are given in Sections \ref{sec:classification}--\ref{sec:10types}. Many of the methods used here, particularly the edge-length graphs described in Section \ref{sec:edgelengraphs}, could be extended to tiles of other $n$-hedra or non-face-to-face tiles (Section \ref{sec:besttile}).

Even the characterization of convex planar tiles was settled only recently with Rao's  \cite{Rao} characterization of pentagonal tilings.

As for general $n$-hedral tiles, Gallagher \emph{et al.} \cite{Gal} give conjectures for all values of $n$. The only previous proofs were for the cube ($n=6$) and the triangular prism ($n=5$).
In both cases, the result for the least-area tile is simply the least-area $n$-hedron, which happens to tile.
This paper presents the first proven case where requiring an $n$-hedron to tile changes the surface area minimizer.

\subsection*{Outline of paper}
Section \ref{sec:tetrahedra} sets terminology and summarizes previously-known facts about tetrahedra that are used in proofs throughout this paper. Section \ref{sec:prevtiles} gives a comprehensive review of known tetrahedral tiles and proves that the Sommerville No. 1 is the least-area tile among known tetrahedral tiles.

Section \ref{sec:edgelengraphs} introduces a combinatorial object called an \emph{edge-length graph}, which is used to deduce linear systems of dihedral angles that are necessary conditions for tetrahedral tiles.

Section \ref{sec:2pi/n} completely identifies all tiling tetrahedra whose dihedral angles are all of the form $2\pi/n$.

Section \ref{sec:classification} classifies all tetrahedra into 25 types. Of these 25 types, the tiling behavior of eleven types has been completely characterized by previous literature. An additional four types are proven not to tile, which was not previously addressed by the literature. Some necessary conditions for the remaining ten types to tile are given in Section \ref{sec:10types}.

Section \ref{sec:boundiso} provides a lower bound on dihedral angles that allows us to search the remaining ten types for a tetrahedral tile with less surface area than the Sommerville No. 1.

Section \ref{sec:code} describes how computer code is used with necessary conditions from Sections \ref{sec:10types} and \ref{sec:boundiso} to reduce the number of candidates to beat the Sommerville No. 1 to seven remaining cases. Section \ref{sec:remainingcases} proves that none of these remaining seven cases tiles with less surface area than the Sommerville No. 1.

Finally, in Section \ref{sec:besttile} we conclude that the Sommerville No. 1 is the least-surface-area tetrahedral tile. We comment on how methods used in this paper may extend to tilings of other polyhedra or non-face-to-face-tiles. We conjecture that the list of tetrahedral tiles identified in previous literature is exhaustive. We also conjecture that a surface-area-minimizing $n$-hedron is identical to its mirror image.

\subsection*{Note on code} 
Many proofs in this paper are computer-assisted.
The code for these computations is available in the GitHub repository at
\url{https://github.com/arjunkakkar8/Tetrahedra}.

\subsection*{Acknowledgements} 
The results presented in this paper came from research conducted by the 2017 Geometry Group at the Williams College SMALL NSF REU. We would like to thank our adviser Frank Morgan for his guidance on this project, as well as Stan Wagon \cite[Chapt. 4]{BLWW} for sending us his code
on reliable computation methods used in this paper.
We would also like to thank the National Science Foundation; Williams College; Michigan State University; University of Maryland, College Park; and Massachusetts Institute of Technology.

\section{Tetrahedra and Tilings}
\label{sec:tetrahedra}
Section \ref{sec:tetrahedra} sets terminology and summarizes previously known facts about tetrahedra that are used in proofs throughout this paper.

We use notation consistent with previous literature. 
Let $T=V_1V_2V_3V_4$ be a tetrahedron where $V_1,V_2,V_3,V_4$ are vertices in $\mb{R}^3$. 
For $\set{i,j,k,l}=\set{1,2,3,4}$, let $F_i$ be the face opposite the vertex $V_i$ with area denoted $\abs{F_i}$. Let $e_{ij}$ be the edge joining $V_i$ and $V_j$ with length denoted $d_{ij}=\abs{e_{ij}}$. (Note that the edge between the faces $F_i$ and $F_j$ is $e_{kl}$, not $e_{ij}$.) Let $\theta_{ij}$ be the dihedral angle between two faces adjacent to the edge $e_{ij}$, i.e. between faces $F_k$ and $F_l$. 

We assume that all tilings are face-to-face.
We make a distinction between orientation-preserving and non-orientation-preserving tilings.
\begin{definition}
\label{def:orientation-preserving}
A tiling is \emph{orientation-preserving} if any two tiles are equivalent under an orientation-preserving isometry of $\R^3$.
\end{definition}
\noindent If a tetrahedron tiles without the use of its mirror image or if the tetrahedron is not distinct from its mirror image, then the resulting tiling is orientation-preserving.
On the other hand, if a tetrahedron tiles with the use of its distinct mirror image, then the resulting tiling is non-orientation-preserving.

\medskip
Lemmas \ref{lem:lawcosine}--\ref{lem:computedihedral} summarize previously known and useful facts.

\begin{lemma}[Law of cosines for a tetrahedron]
\label{lem:lawcosine}
$$\abs{F_1}^2 = \abs{F_2}^2 + \abs{F_3}^2 + \abs{F_4}^2
- 2\paren{\abs{F_2}\abs{F_3} \cos \theta_{14} +\abs{F_2}\abs{F_4} \cos \theta_{13} +\abs{F_3}\abs{F_4} \cos \theta_{12}}.$$
\end{lemma}

Lemma \ref{lem:dihedralrelation} relates the areas of the faces to the dihedral angles between them.  Geometrically, it says that the area of the base face $F_1$ is equal to the sum of the areas of the projections of the remaining faces. 

\begin{lemma}
\label{lem:dihedralrelation}
$$\abs{F_1}=\abs{F_2}\cos\theta_{34}+\abs{F_3}\cos\theta_{24}+\abs{F_4} \cos\theta_{23}.$$
Moreover, if $c_{ij}=\cos\theta_{ij}$, then:
$$
\begin{vmatrix}
-1 & c_{34} & c_{24} & c_{23} \\ 
c_{34} & -1 & c_{14} & c_{13} \\ 
c_{24} & c_{14} & -1 & c_{12} \\ 
c_{23} & c_{13} & c_{12} & -1
\end{vmatrix}=0,
$$
and the area vector $(\abs{F_1}, \abs{F_2}, \abs{F_3}, \abs{F_4})^T$
is in the null space of the matrix corresponding to this determinant.
\end{lemma}
 The following lemmas from Wirth and Dreiding {\cite{Wir}} impose constraints on the dihedral angles and edge lengths of a tetrahedron.
\begin{lemma}[{\cite[Lemmas 1 and 3]{Wir}}]
\label{lem:dihedralineq}
$$\Omega_1=\theta_{12}+\theta_{13}+\theta_{14}-\pi>0,$$
where $\Omega_1$ is the solid angle at the vertex $V_1$ and
$$\theta_{13} + \theta_{14} + \theta_{24} + \theta_{23} < 2\pi.$$
\end{lemma}

\begin{lemma}[{\cite[Lemma 4 and Rmk. 3]{Wir}}]
\label{lem:edgecond}
A sextuple $(d_{12}, d_{13}, d_{14}, d_{23}, d_{24}, d_{34})$
provides the edge lengths of a tetrahedron if and only if
\begin{enumerate}
\item For any distinct $i,j,k \in \set{1,2,3,4}$,
the triangle inequality $d_{ij} < d_{jk}+d_{ik}$ holds,
where we interpret $d_{ij}=d_{ji}$, and
\item the \emph{Cayley-Menger determinant} is positive:
$$
D = \begin{vmatrix}
0 & d_{12}^2 & d_{13}^2 & d_{14}^2 & 1 \\ 
d_{12}^2 & 0 & d_{23}^2 & d_{24}^2 & 1 \\ 
d_{13}^2 & d_{23}^2 & 0 & d_{34}^2 & 1 \\ 
d_{14}^2 & d_{24}^2 & d_{34}^2 & 0 & 1 \\
1 & 1 & 1 & 1 & 0
\end{vmatrix}
> 0.
$$
\end{enumerate}
Moreover, the volume $V$ of the tetrahedron can be written
in terms of this determinant as
$$D=288V^2.$$
\end{lemma}

\begin{lemma}[{\cite[Theorem 1]{Wir}}]
\label{lem:computedihedral}
The dihedral angle $\theta_{ij}$ is given by
$$\cos \theta_{ij}=\frac{D_{ij}}{\sqrt{D_{ijk}D_{ijl}}},$$
where
$$D_{ij}=-d_{ij}^4+
\paren{d_{ik}^2+d_{il}^2+d_{jk}^2+d_{jl}^2-2d_{kl}^2}d_{ij}^2
+\paren{d_{ik}^2-d_{jk}^2}\paren{d_{jl}^2-d_{il}^2}$$
and
$$D_{ijk}=-16\abs{F_l}^2.$$
\end{lemma}

The next proposition says that the dihedral angles of a tetrahedron determine the tetrahedron.

\begin{prop}
\label{prop:dihedralfixes}
The dihedral angles (order matters) determine the tetrahedron
up to translation, rotation, and scaling.
\end{prop}

\begin{proof}
A tetrahedron is made up of four planes. Specifying the dihedral
angles means that the angles between all pairs of planes are specified.

Consider the incomplete tetrahedron with three planes intersecting at
a vertex at three specified dihedral angles. This configuration is uniquely
determined up to translation and rotation.
Then, the fourth plane has a fixed orientation with respect
to the incomplete tetrahedron, and its translation corresponds
to scaling the tetrahedron. Hence, the tetrahedron is fixed
up to translation, rotation and scaling.
\end{proof}

\section{Previous Literature}
\label{sec:prevtiles}
Section \ref{sec:prevtiles} summarizes previous results on tetrahedral tiles.

\subsection*{Sommerville}

Sommerville \cite{Solist} provided a list of every orientation-preserving
face-to-face tetrahedral tile, which was proven to be complete by Edmonds \cite{Ed}. 
In contrast, our definition of space-filling also includes non-orientation-preserving tiles. The tetrahedra Sommerville dismissed because of being non-orientation preserving are considered in this paper, but they are not labeled by Sommerville's nomenclature. A full description of the types considered in this paper is given in Section \ref{sec:classification}.

Out of eleven total candidates, Sommerville proved four of them, named Sommerville No. 1--4 by Goldberg \cite{Gol}, tiled Euclidean space in an orientation-preserving manner. Gallagher \emph{et al.} \cite{Gal} computed that the surface-area minimizing tetrahedron among these is the Sommerville No. 1 (Figure \ref{fig:sommerville1}), which has four congruent isosceles sides with two edges of length $\sqrt{3}$ and one edge of length $2$. It can be obtained by slicing a triangular prism into three congruent tetrahedra
(Figure \ref{fig:sommerville1}).

Edmonds pointed out that the Sommerville (xi) tetrahedron 
(with $d_{12}=d_{23}=d_{34}=p$ and $d_{13}=d_{14}=d_{24}=q$) 
was neglected in Sommerville's original paper, but proved that it sometimes tiles face-to-face,
although in a non-orientation-preserving manner.
We present this result as Proposition \ref{prop:somxi_tile}, which is used several times throughout our paper.

\begin{prop}
\label{prop:somxi_tile}
The Sommerville (xi) tetrahedron (with $d_{12}=d_{23}=d_{34}=p$ and $d_{13}=d_{14}=d_{24}=q$) tiles face-to-face
precisely when $p/q=\sqrt{2/3}$ or $\sqrt{3/2}$.
However, it tiles only in a non-orientation-preserving manner.
\end{prop}

\begin{proof}
Proven by Edmonds \cite{Ed}.
\end{proof}



\subsection*{Hill and Baumgartner}
Hill and Baumgartner \cite{Hi, Ba1, Ba} both contributed to early work in tetrahedral tilings. We generally do not refer to their tetrahedra by name, since all except Baumgartner's T2 (Hill's second type) are subcases of Sommerville No. 1--4. Baumgartner's T2 is a subcase of Goldberg's second and third infinite families (Prop. \ref{prop:prevtetrahedra}). A summary of Hill's, Baumgartner's, and Sommerville's work is provided by Goldberg \cite{Gol}.

\subsection*{Goldberg}
Following papers of Sommerville and Baumgartner, Goldberg \cite{Gol} presented three infinite families of tetrahedral tiles
(Figs. \ref{fig:goldbergfamily1}--\ref{fig:goldbergfamily3}), which in general need reflections to tile.
Sommerville No. 1 and No. 3 appear as degenerate cases of tiles which are
their own reflections.

\begin{figure}[h!]
\begin{tikzpicture}[scale=1.5,transform shape]
\draw [black] (0,3) node [anchor=south] {\scriptsize 1} -- (-2.4,0) node [anchor=east] {\scriptsize 2};
\node [anchor=south east] at (-1,1.4) {\small $b,\alpha$};
\draw [black] (0,3) -- (0,-1) node [anchor=north] {\scriptsize 3};
\node [anchor=center] at (.2,1.2) {\small $c,\pi/2$};
\draw [black] (0,3) -- (2.4,0) node [anchor=west] {\scriptsize 4};
\node [anchor=south west] at (1,1.4) {\small $a,\pi/3$};
\draw [black] (-2.4,0) -- (0,-1);
\node [anchor=north east] at (-.7,-.5) {\small $b, \pi-2\alpha$};
\draw [dashed, black] (-2.4,0) -- (2.4,0);
\node [anchor=south] at (.9,-.2) {\small $c,\pi/2$};
\draw [black] (0,-1) -- (2.4,0);
\node [anchor=north west] at (.8,-.5) {\small $b,\alpha$};
\end{tikzpicture}
\begin{align*}
a^2/3+b^2&=c^2 \\
\sin\alpha &=b/c
\end{align*}
\caption{Goldberg's first family labeled with edge lengths and dihedral angles.}
\label{fig:goldbergfamily1}
\end{figure}

\begin{figure}[h!]
\begin{tikzpicture}[scale=1.5,transform shape]
\draw [black] (0,3) node [anchor=south] {\scriptsize 1} -- (-2.4,0) node [anchor=east] {\scriptsize 2};
\node [anchor=south east] at (-1,1.4) {\small $b,\alpha$};
\draw [black] (0,3) -- (0,-1) node [anchor=north] {\scriptsize 3};
\node [anchor=center] at (.2,1.2) {\small $c,\pi/2$};
\draw [black] (0,3) -- (2.4,0) node [anchor=west] {\scriptsize 4};
\node [anchor=south west] at (1,1.4) {\small $a/2,\pi/3$};
\draw [black] (-2.4,0) -- (0,-1);
\node [anchor=north east] at (-.7,-.5) {\small $b, \pi/2-\alpha$};
\draw [dashed, black] (-2.4,0) -- (2.4,0);
\node [anchor=south] at (1,-.2) {\small $d, \pi/2 + \beta$};
\draw [black] (0,-1) -- (2.4,0);
\node [anchor=north west] at (.8,-.5) {\small $d,\pi/2-\beta$};
\end{tikzpicture}
\begin{align*}
a^2/12+d^2 &=b^2\\
a^2/3+b^2&=c^2\\
\sin\alpha&=b/c\\
\sin\beta &= a/(2\sqrt{3}b)
\end{align*}
\caption{Goldberg's second family labeled with edge lengths and dihedral angles.}
\label{fig:goldbergfamily2}
\end{figure}

\begin{figure}[h!]
\begin{tikzpicture}[scale=1.5,transform shape]
\draw [black] (0,3) node [anchor=south] {\scriptsize 1} -- (-2.4,0) node [anchor=east] {\scriptsize 2};
\node [anchor=south east] at (-1,1.4) {\small $b,\alpha$};
\draw [black] (0,3) -- (0,-1) node [anchor=north] {\scriptsize 3};
\node [anchor=center] at (.2,1.2) {\small $a,\pi/6$};
\draw [black] (0,3) -- (2.4,0) node [anchor=west] {\scriptsize 4};
\node [anchor=south west] at (1,1.4) {\small $f, \pi/2 + \gamma$};
\draw [black] (-2.4,0) -- (0,-1);
\node [anchor=north east] at (-.7,-.5) {\small $c, \pi/2$};
\draw [dashed, black] (-2.4,0) -- (2.4,0);
\node [anchor=south] at (1.1,-.2) {\small $b/2,\pi-2\alpha$};
\draw [black] (0,-1) -- (2.4,0);
\node [anchor=north west] at (.8,-.5) {\small $f,\pi/2-\gamma$};
\end{tikzpicture}
\begin{align*}
b^2+2c^2&=4f^2\\
a^2/3+b^2&=c^2\\
\sin\alpha&=b/c\\
\sin\beta &= a/(2\sqrt{3}b)\\
\sin\gamma &=a/(3c)
\end{align*}
\caption{Goldberg's third family labeled with edge lengths and dihedral angles.}
\label{fig:goldbergfamily3}
\end{figure}

Goldberg also considered non-face-to-face tiles.
In fact, Goldberg's first and second families tile face-to-face
while Goldberg's third family tiles only non-face-to-face in general.

\begin{prop}
\label{prop:goldberg12tile}
Goldberg's first and second tiling families can be face-to-face
but not necessarily orientation-preserving.
\end{prop}

\begin{proof}
Consider Goldberg's first family.
Goldberg derives it by dissecting an equilateral triangular prism into tetrahedra that are congruent without reflection as in Figure \ref{fig:goldbergslice}. The tetrahedra may be stacked into an arbitrarily long prism.
Each face of the prism is composed of triangles with side lengths
$3a$, $b$ and $c$, and all faces are congruent without reflection,
as illustrated in Figure \ref{fig:goldbergtile}.

\begin{figure}[ht]
\includegraphics[width = 0.5\linewidth]{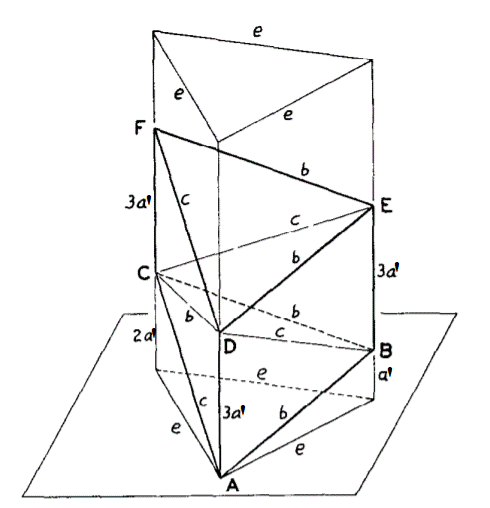}
\caption{Derivation of Goldberg's first family \cite{Gol}. We have primed $a$ in Goldberg's image; in our paper, $a=3a'$. }
\label{fig:goldbergslice}
\end{figure}

\begin{figure}
\centering
\begin{subfigure}{.5\textwidth}
\begin{tikzpicture}[scale=1.3,transform shape]
\draw [black] (-2.4,2) -- (-2.4,-2);
\draw [black] (-0.5, 2) -- (-0.5, -2);
\draw [black, dashed] (-2.4, 2) -- (-0.5, 2);
\draw [black, dashed] (-2.4, -2) -- (-0.5, -2);
\draw [red, dashed] (-2.4, 2) -- (-0.5, 1);
\draw [red, dashed] (-2.4, -1) -- (-0.5, 1);
\draw [red, dashed] (-2.4, -1) -- (-0.5, -2);

\draw [black] (2.4,2) -- (2.4,-2);
\draw [black] (0.5, 2) -- (0.5, -2);
\draw [black, dashed, thick] (2.4, 2) -- (0.5, 2);
\draw [black, dashed, thick] (2.4, -2) -- (0.5, -2);
\draw [red, dashed] (2.4, 2) -- (0.5, 1);
\draw [red, dashed] (2.4, -1) -- (0.5, 1);
\draw [red, dashed] (2.4, -1) -- (0.5, -2);

\draw [->, thick] (-1.45, 2) to (-1.45, 2.3);
\draw [->, thick] (-1.45, -2) to (-1.45, -2.3);
\draw [->, thick] (1.45, 2) to (1.45, 2.3);
\draw [->, thick] (1.45, -2) to (1.45, -2.3);

\end{tikzpicture}
   \centering
   \caption{}
   \label{fig:goldbergtile_sub1}
\end{subfigure}
\begin{subfigure}{.4\textwidth}
\begin{tikzpicture}[scale=2.5,transform shape]
\draw [black] (0,0) -- (1,0);
\draw [black] (0,0) -- (0.5, 0.866);
\draw [black] (1,0) -- (0.5, 0.866);

\draw [black] (0,0) -- (0.5, -0.866);
\draw [black] (0,0) -- (-0.5,-0.866);
\draw [black] (-0.5, -0.866) -- (0.5, -0.866);

\draw [black] (0,0) -- (-0.5,0.866);
\draw [black] (0,0) -- (-1,0);
\draw [black] (-0.5,0.866) -- (-1,0);

\draw [black] (0.5, 0.866) -- (-0.5,0.866);
\draw [black] (-1,0) -- (-0.5, -0.866);
\draw [black] (0.5, -0.866) -- (1,0);

\node [anchor=center] at (0, 0.5){\tiny $A'$};
\node [anchor=center] at (0, -0.5){\tiny $A$};
\node [anchor=center] at (0.5, 0.3){\tiny $A$};
\node [anchor=center] at (0.5, -0.3){\tiny $A'$};
\node [anchor=center] at (-0.5, 0.3){\tiny $A$};
\node [anchor=center] at (-0.5, -0.3){\tiny $A'$};
\end{tikzpicture}
  \centering
  \caption{}
  \label{fig:goldbergtile_sub2}
\end{subfigure}
\caption{(\ref{fig:goldbergtile_sub1}): The rectangular section of the face of a prism sliced into tetrahedra that are congruent without reflection (left). In order for the tetrahedra within to tile face-to-face, a reflection of the prism (right) must be used. (\ref{fig:goldbergtile_sub2}): Tiling of Goldberg prisms. $A$ denotes the original prism, and $A'$ is the mirror image. The prism and its reflection tile space face-to-face, proving that Goldberg's first family tiles face-to-face.}
\label{fig:goldbergtile}
\end{figure}
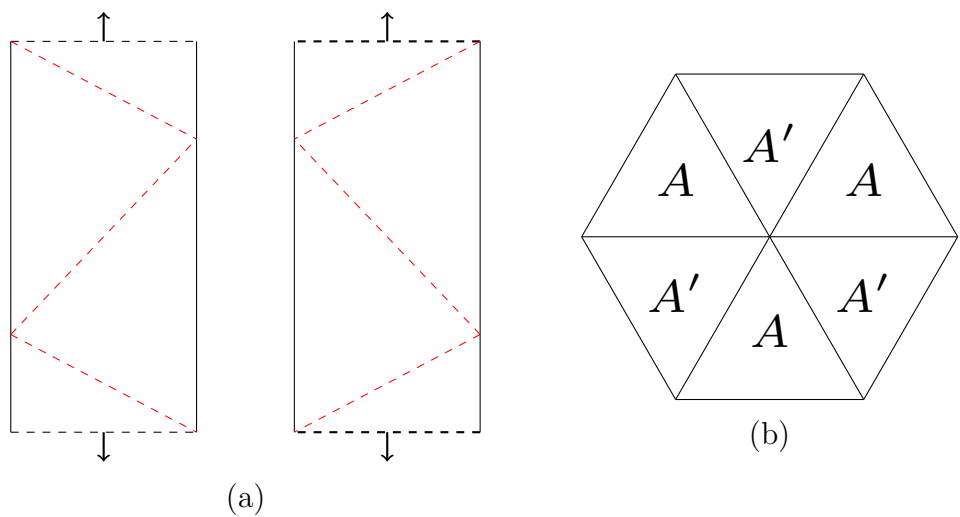

Since the prism has an equilateral triangular cross section, it can tile space.
In order for the tetrahedra within to tile face-to-face, a reflection of the prism must also be used (except for cases where the tetrahedron does not have a distinct reflection), as illustrated in Figure \ref{fig:goldbergtile}.
Then the prism and its reflection can be arranged as in Figure \ref{fig:goldbergtile_sub2} to tile space, and the tetrahedra within tile face-to-face.

Goldberg's second family is derived from slicing
the tetrahedra in Goldberg's first family in half, as in Figure
\ref{fig:goldberg2ndderive}.
To show that it can tile face-to-face,
notice that this slicing cuts the triangles on the face of the prism
in half in such a way that the prism and its reflection can still
match up face-to-face as in Figure \ref{fig:goldbergtile}.

\begin{figure}[ht]
\begin{tikzpicture}[scale=1.5,transform shape]
\draw [black] (0,3) node [anchor=south] {\scriptsize C} -- (-2.6,0.6) node [anchor=east] {\scriptsize A};
\draw [black] (0,3) -- (-0.8,-1) node [anchor=north] {\scriptsize B};
\draw [black] (0,3) -- (2.2,-0.2) node [anchor=west] {\scriptsize D};
\draw [black] (-2.6,0.6) -- (-0.8,-1);
\draw [dashed, black] (-2.6,0.6) -- (2.2,-0.2);
\draw [black] (-0.8,-1) -- (2.2,-0.2);
\draw [dashed, black] (0,3) -- (-0.2,0.1);
\draw [dashed, black] (-0.8,-1) -- (-0.2,0.1);
\node [anchor=north west] at (-0.3,0.2) {\scriptsize G};
\end{tikzpicture}
\caption{Goldberg's second family is derived by slicing a tetrahedron $ABCD$ of the first family into congruent halves $ACBG$ and $BCDG$. Vertex labeling corresponds to Goldberg's own labeling (Figure \ref{fig:goldbergslice}).}
\label{fig:goldberg2ndderive}
\end{figure}
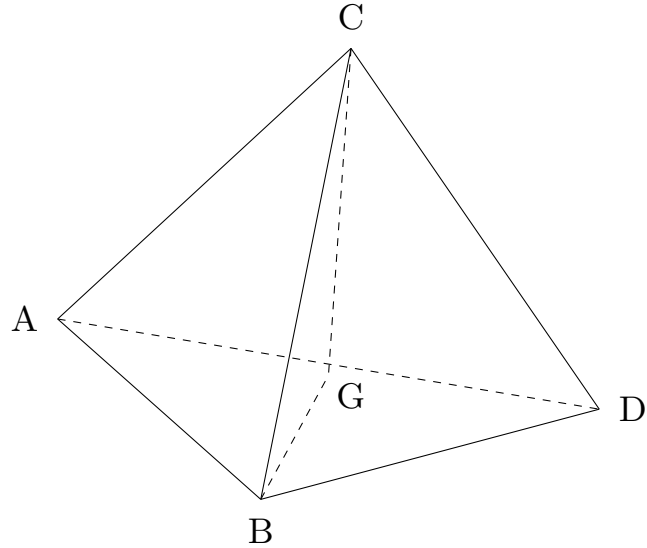

\end{proof}

\begin{remark}
Note that by reflecting the prism as in Figure \ref{fig:goldbergtile}, the tetrahedra within are also reflected, and the tile is no longer orientation-preserving in general.
\end{remark}

\begin{prop}
Goldberg's third family does not tile face-to-face in general.
\end{prop}

\begin{proof}
In general, a tetrahedron in Goldberg's third family
has only one edge of length $b$ which has dihedral angle $\alpha$ which can vary
continuously. In a face-to-face tile, the tetrahedra must stack
around the edge of length $b$, and the dihedral angles
around that edge must sum to $2\pi$. This means that some multiple of $\alpha$
must be $2\pi$, so $\alpha$ cannot vary continuously.
\end{proof}

We now show that the least-area tetrahedral tile among previously identified tiles is the Sommerville No. 1. We start by showing that the Goldberg
families include many of the previous tiles.

\begin{prop}
\label{prop:prevtetrahedra}
All previously identified tetrahedral tiles
except Sommerville No. 2 and Sommerville No. 4 are included in
Goldberg's families.
\end{prop}

\begin{proof}
We consider each previously mentioned tetrahedron.
Many of these justifications are also made by Goldberg \cite{Gol}.

\begin{enumerate}
\item Sommerville No. 1: First Goldberg family with $\alpha=\pi/3$.
\item Sommerville No. 2: Not in the Goldberg families.
\item Sommerville No. 3: Second Goldberg family with $\alpha=\pi/4$.
\item Sommerville No. 4: Not in the Goldberg families.
\item Sommerville (xi), $p/q=\sqrt{2/3}$: First Goldberg family with $3a=c$.
\item Baumgartner T2: Second Goldberg family with $\alpha=\pi/3$.
Also third Goldberg family with $\alpha=\pi/3$.
\end{enumerate}
\end{proof}

Now we identify the least-area
tetrahedron among Goldberg's infinite families.

\begin{theorem}
\label{thm:goldbergmin}
The least-area tetrahedron among Goldberg's infinite families is the Sommerville No. 1.
\end{theorem}

\begin{proof}
Let $V$ be the volume of a tetrahedron and $S$ be its surface area.
To minimize the surface area given unit volume,
we  minimize the quantity $S^3/V^2$, which is invariant up to scaling.
Given edge lengths, the volume can be written in terms of the Cayley-Menger determinant as in Lemma \ref{lem:edgecond}.
The surface area is the sum of the face areas calculated
using Heron's formula:
\begin{align*}
S=&\sqrt{s_1(s_1-d_{23})(s_1-d_{24})(s_1-d_{34})}+\sqrt{s_2(s_2-d_{13})(s_2-d_{14})(s_2-d_{34})}\\
+&\sqrt{s_3(s_3-d_{12})(s_3-d_{14})(s_3-d_{24})}+\sqrt{s_4(s_4-d_{12})(s_4-d_{13})(s_4-d_{23})},
\end{align*}
where
$$s_i=\frac{d_{jk}+d_{kl}+d_{jl}}{2}$$
for $\set{i,j,k,l}=\set{1,2,3,4}$.
Each of Goldberg's families have one degree of freedom
and therefore $S^3/V^2$ can be written as a function of one parameter.
We then find the minimum of this function with Mathematica.

\begin{enumerate}
\item First family: $d_{12}=1, \ d_{13}=\sqrt{3a^2+1}, \ d_{14}=3a, \ d_{23}=1, \ d_{24}=\sqrt{3a^2+1}, \ d_{34}=1$.

The minimum is $\sim 7.413$ at $a=1/3$.
\item Second family: $d_{12}=1, \ d_{13}=\sqrt{3a^2+1}, \ d_{14}=3a/2, \ d_{23}=1, \ d_{24}=\sqrt{4-3a^2}/2, \ d_{34}=\sqrt{4-3a^2}/2$.

The minimum is $\sim 8.109$ at $a\approx0.491$.
\item Third family: $d_{12}=1, \ d_{13}=\sqrt{6a^2+3}/2, \ d_{14}=3a, \ d_{23}=1/2, \ d_{24}=\sqrt{3a^2+1}, \ d_{34}=\sqrt{6a^2+3}/2$.

The minimum is $\sim 8.273$ at $a\approx0.238$.
\end{enumerate}

From these calculations, Goldberg's first family with $a=1/3$ is
the least-area tetrahedron, and this tetrahedron is the Sommerville No. 1.
The code for this computation is in the GitHub repository in the file
\verb|GoldbergTetrahedra.nb|.
\end{proof}

From Proposition \ref{prop:prevtetrahedra} and Theorem \ref{thm:goldbergmin} we can deduce the least-area tetrahedral tile among all previously identified tiles.

\begin{cor}
\label{cor:prevknown_so1best}
The least-area tile among all previously identified tetrahedral tiles is the Sommerville No. 1.
\end{cor}

\begin{proof}
By Proposition \ref{prop:prevtetrahedra}, all previously identified
tetrahedral tiles are Sommerville No. 2 and No. 4 and
the Goldberg infinite families. Gallagher \emph{et al.} \cite{Gal} calculated that the Sommerville No. 1 is the least-area tetrahedron among all the Sommerville cases,
and by Theorem \ref{thm:goldbergmin} the Sommerville No. 1 is the least-area tetrahedron among the Goldberg infinite families.
\end{proof}

\section{Edge-Length Graphs}
\label{sec:edgelengraphs}
In Section \ref{sec:edgelengraphs} we define a combinatorial object called an
\emph{edge-length graph}, which is used to show that certain linear combinations of dihedral angles sum to $2\pi$ (Prop. \ref{prop:edgegraphlincomb}). We continue to assume that tilings are face-to-face.
Here we consider only tetrahedra, but this concept can be easily generalized
to other polyhedra.

\begin{definition}
Let $T$ be a tetrahedron and $d$ one of its edge lengths. We define the
\emph{$d$-edge-length graph} to be the graph where for
$\set{i,j,k,l}=\set{1,2,3,4}$, if edge $e_{ij}$ has length $d$,
the ordered pairs of edge lengths $(d_{ik}, d_{jk})$
and $(d_{il}, d_{jl})$ are an edge of the graph.
We label this graph edge with the unordered pair of numbers $\set{i,j}$.

Notice that if $(d_{ik}, d_{jk})=(d_{il}, d_{jl})$, then this graph edge
is a loop. Each graph edge can have multiple labels.
\end{definition}

The meaning of the edge-length graph is given in Lemma \ref{lem:edgegraphmean}, which states that going from node to node in the edge-length graph is equivalent to stacking a tetrahedron around an edge. This concept is illustrated in Figure \ref{fig:stackaroundedge}.

\begin{figure}[ht]
\begin{subfigure}{.5\textwidth}
\begin{tikzpicture}[scale=1.5,transform shape]
\draw [black] (0,2) -- (0,-2);
\draw [black] (-1.5,-0.3) -- (1.5,0.3);
\draw [black] (0,2) -- (-1.5,-0.3);
\draw [black] (0,-2) -- (-1.5,-0.3);
\draw [black] (0,2) -- (1.5,0.3);
\draw [black] (0,-2) -- (1.5,0.3);

\draw [black, dashed] (-1.5,-0.3) -- (-1.3, 0.4);
\draw [black, dashed] (-1.3, 0.4) -- (0,2);
\draw [black, dashed] (-1.3, 0.4) -- (0,-2);

\draw [black, dashed] (1.5,0.3) -- (1.7, -0.4);
\draw [black, dashed] (1.7, -0.4) -- (0,2);
\draw [black, dashed] (1.7, -0.4) -- (0,-2);


\node [anchor=south] at (0, 2){\small 1};
\node [anchor=north] at (0, -2){\small 3};
\node [anchor=west] at (1.5,0.3){\small 4};
\node [anchor=east] at (-1.5,-0.3){\small 2};

\node [anchor=center] at (-0.1,-0.7){\small $a$};
\node [anchor=center] at (0.5,0.25){\small $f$};
\node [anchor=center] at (1,1.1){\small $d$};
\node [anchor=center] at (-0.6,0.85){\small $b$};
\node [anchor=center] at (0.6,-0.9){\small $e$};
\node [anchor=center] at (-1,-1.1){\small $c$};
\end{tikzpicture}

$$\pathlentwo{(b,c)}{(d,e)}{13}{.75}{1}$$
$$\pathlentwo{(c,b)}{(e,d)}{13}{.75}{1}$$

   \centering
   \caption{Side View}
\end{subfigure}
\begin{subfigure}{.4\textwidth}
\begin{tikzpicture}[scale=2,transform shape]
\draw [black] (0,1) -- (1,1);
\draw [black] (0,1) -- (0.5, 1.866);
\draw [black] (1,1) -- (0.5, 1.866);

\draw [black] (0,1) -- (0.5, 0.134);
\draw [black] (0,1) -- (-0.5, 0.134);
\draw [black] (-0.5, 0.134) -- (0.5, 0.134);
\draw [black] (0,1) -- (-1,1);
\draw [black] (-1,1) -- (-0.5, 0.134);
\draw [black] (0.5, 0.134) -- (1,1);

\node [anchor=center] at (-0.05, 1.1) {\tiny $a$};
\node [anchor=center] at (0.85, 1.45) {\tiny $f$};
\node [anchor=center] at (0.85, 0.55) {\tiny $f$};
\node [anchor=center] at (-0.85, 0.55) {\tiny $f$};
\node [anchor=center] at (0, 0) {\tiny $f$};
\node [anchor=center] at (-0.5, 1.1) {\tiny $b$};
\node [anchor=center] at (0.5, 1.1) {\tiny $d$};
\node [anchor=center] at (-0.25, 0.5) {\tiny $d$};
\node [anchor=center] at (0.25, 0.5) {\tiny $b$};
\node [anchor=center] at (0.25, 1.55) {\tiny $b$};

\draw [->, thick] (0.2,1.8) to [out=160, in=90](-0.7, 1.1);

\node [anchor=center] at (0, -0.3) {\tiny From Vertex 1};

\draw [black] (0,-1.4) -- (1,-1.4);
\draw [black] (0,-1.4) -- (0.5, -0.534);
\draw [black] (1,-1.4) -- (0.5, -0.534);

\draw [black] (0,-1.4) -- (0.5, -2.266);
\draw [black] (0,-1.4) -- (-0.5,-2.266);
\draw [black] (-0.5, -2.266) -- (0.5, -2.266);
\draw [black] (0,-1.4) -- (-1,-1.4);
\draw [black] (-1,-1.4) -- (-0.5, -2.266);
\draw [black] (0.5, -2.266) -- (1,-1.4);

\node [anchor=center] at (-0.05, -1.3) {\tiny $a$};
\node [anchor=center] at (0.85, -0.95) {\tiny $f$};
\node [anchor=center] at (0.85, -1.85) {\tiny $f$};
\node [anchor=center] at (-0.85, -1.85) {\tiny $f$};
\node [anchor=center] at (0, -2.4) {\tiny $f$};
\node [anchor=center] at (-0.5, -1.3) {\tiny $c$};
\node [anchor=center] at (0.5, -1.3) {\tiny $e$};
\node [anchor=center] at (-0.25, -1.9) {\tiny $e$};
\node [anchor=center] at (0.2, -1.9) {\tiny $c$};
\node [anchor=center] at (0.2, -0.95) {\tiny $c$};

\draw [->, thick] (0.2,-0.6) to [out=160, in=90](-0.7, -1.3);

\node [anchor=center] at (0, -2.75) {\tiny From Vertex 3};

\end{tikzpicture}
  \centering
  \caption{Top and Bottom View}
\end{subfigure}
\caption{The $a$-edge-length graph for a tetrahedron with edge lengths $a,b,c,d,e,f$ all distinct. The first line describes that the edge with length $a$, $e_{13}$, connects $(b,c)$ to $(d,e)$. The second line describes that the edge with length $a$, $e_{13}$, connects $(c,b)$ to $(e,d)$. From this graph and Proposition \ref{prop:edgegraphlincomb}, $n\theta_{13}=2\pi$ for some $n\in\mb{N}$.}
\label{fig:stackaroundedge}
\end{figure}
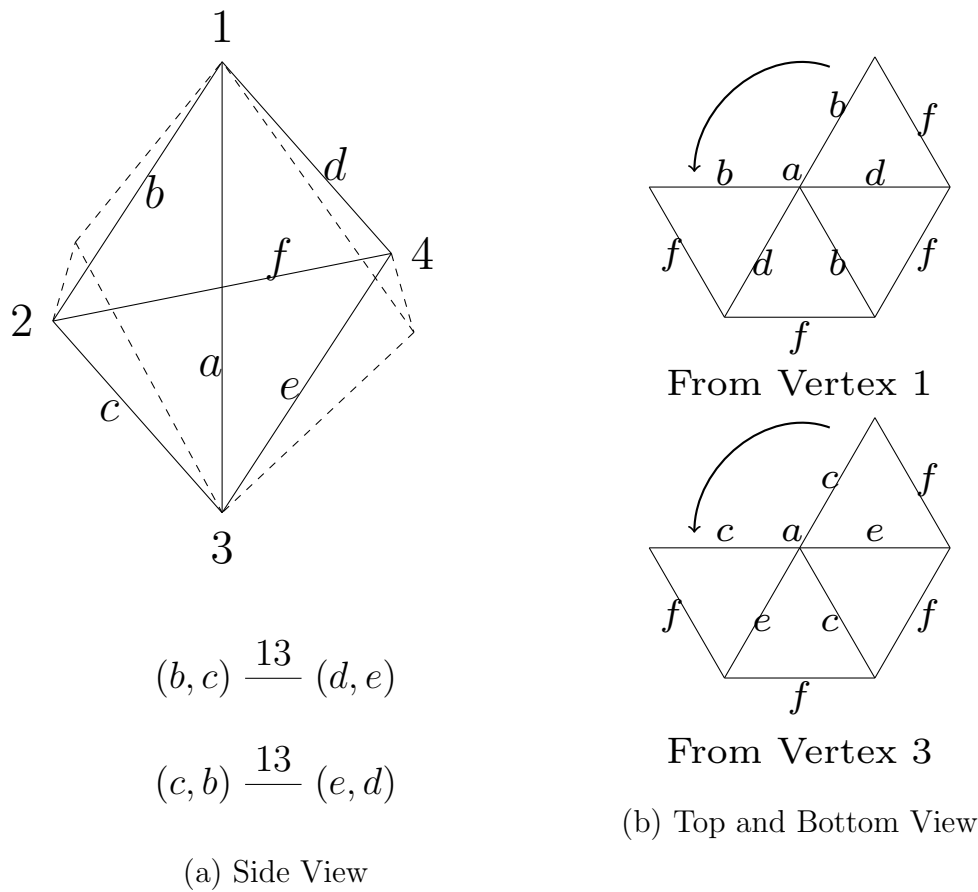

\begin{lemma}
\label{lem:edgegraphmean}
Let $T$ be a tetrahedron.
Suppose you have disjoint copies $T_i$ of $T$ stacked around a line segment $PQ$
(with one edge of each copy coinciding with $PQ$,
so that the sum of the dihedral angles is $2\pi$).
Copy $T_i$ of $T$ has faces $PQR_i$ and $PQR_{i+1}$.
Then $(PR_i,QR_i)$ and $(PR_{i+1},QR_{i+1})$ are an edge
in the $d$-edge-length graph of $T$.
In other words,
$$(PR_1,QR_1),\dots,(PR_n,QR_n),(PR_1,QR_1)$$
is a closed walk in the $d$-edge-length graph of $T$.
\end{lemma}

\begin{proof}
This follows directly from the definition of the edge-length graph.
\end{proof}

The following proposition provides information on linear
combinations of dihedral angles from the edge-length graph.

\begin{prop}
\label{prop:edgegraphlincomb}
Let $T$ be a tetrahedral tile and $d$ one of its edge lengths.
From any node in the $d$-edge-length graph of $T$,
there are a closed walk of length $n$ starting and ending at that node
and labels $\set{i_1,j_1},\set{i_2,j_2},\dots,\set{i_n,j_n}$,
where $\set{i_k,j_k}$ is a label of the $k$th edge of the closed walk,
such that
$$\theta_{i_1j_1}+\dots+\theta_{i_nj_n}=2\pi.$$
\end{prop}

\begin{proof}
Any node of the edge-length graph corresponds to at least one pair of adjacent edges of a face of $T$. Fix a tiling, and consider how $T$ stacks around the third edge of the face.
By Lemma \ref{lem:edgegraphmean}, there are a closed walk starting and ending
at that node and labels $\set{i_k,j_k}$ associated with the $k$th edge
of the closed walk such that the angles $\theta_{i_kj_k}$ are the dihedral angles
between copies of $T$ in the stack, summing to $2\pi$.
\end{proof}

\begin{cor}
\label{cor:edgegraphlincomb}
Given an edge length $d$, for all edges $e_{ij}$ of length $d$,
there are nonnegative integers $n_{ij}$ such that
$$\sum n_{ij}\theta_{ij}=2\pi.$$
That is, some linear combination of the dihedral angles associated with edges of length $d$ sums to $2\pi$.
Moreover, there is a closed walk in the $d$-edge-length graph
such that the label $\set{i,j}$ is passed through exactly $n_{ij}$
times, where for a graph edge with multiple labels, we can pick any label
to represent that edge each time we pass through the edge.
\end{cor}

The following corollary is a special case
of Corollary \ref{cor:edgegraphlincomb}.

\begin{cor}
\label{cor:edge2pi/n}
Let $T$ be a tetrahedron which tiles space face-to-face and $d$
one of its edge lengths.
If all edges with length $d$ have the same associated dihedral angle $\theta(d)$,
then $\theta(d)$ is $2\pi/n$ for some $n \in \mathbb{N}$.

Moreover, if the $d$-edge-length graph of $T$ has no closed walk
of odd length,
then $\theta(d)$ is $\pi/n$ for some $n\in\mathbb{N}$.
\end{cor}

\begin{proof}
By Corollary \ref{cor:edgegraphlincomb},
$$\paren{\sum n_{ij}} \theta(d)=2\pi,$$
so $\theta(d)$ is $2\pi/n$ for some $n\in\mb{N}$.
If the $d$-edge-length graph has no closed walk of odd length,
then $\sum n_{ij}$ is even, so
$\theta(d)=\pi/n$ for some $n\in\mb{N}$.
\end{proof}

\begin{remark}
\label{rem:closedwalk}
We can characterize closed walks in specific graphs.
In a tree, a closed walk passes through each edge an even number of times.
If a graph includes a 3-cycle whose removal leaves the three vertices pairwise disconnected, then the numbers of times the edges of this 3-cycle are traversed by a closed walk are either all even or all odd.
\end{remark}

\section{Tetrahedra with Dihedral Angles $2\pi/n$}
\label{sec:2pi/n}
In Section \ref{sec:2pi/n}  Table \ref{tab:2pi/n}, we identify all tetrahedra with 
dihedral angles all of the form $2\pi/n$ in Table \ref{tab:2pi/n}. (We continue to assume that all tilings are face-to-face.)
By Proposition \ref{prop:dihedralfixes}, the (ordered) dihedral angles
determine the tetrahedron up to translation, rotation, and scaling. (Reflection alters the order of the dihedral angles.)
We prove that these lists are exhaustive
in Theorems \ref{thm:2pi/nall} and \ref{thm:2pi/ntile}.

\begin{table}
  \caption{All tetrahedra with dihedral angles of the form $\theta_{ij}=2\pi/n_{ij}$, up to permutation of vertices (Thm. \ref{thm:2pi/nall}). Area denotes the surface area for a tetrahedron of volume 1. The regular tetrahedron has area about 7.21.
  The tiles were previously identified by Sommerville \cite{So} and Goldberg \cite{Gol} (Sect. \ref{sec:prevtiles}). The last six do not tile (Thm. \ref{thm:2pi/ntile}).
}
  \begin{tabular}{| l | c | c | c | c | c | c | c |}
    \hline
    \multicolumn{8}{|c|}{Tiles} \\ \hline
     & $n_{12}$ & $n_{13}$ & $n_{14}$ & $n_{23}$ & $n_{24}$ & $n_{34}$ & Area \\ \hline
    Sommerville No. 3 & 3 & 6 & 6 & 8 & 8 & 4 & 8.18 \\ \hline
    Sommerville No. 2 & 4 & 4 & 4 & 6 & 6 & 8 & 7.96 \\ \hline
    First Goldberg family, $\alpha=2\pi/8$ & 4 & 4 & 8 & 8 & 4 & 6 & 7.97 \\ \hline
    First Goldberg family, $\alpha=2\pi/5$ & 4 & 5 & 6 & 10 & 5 & 4 & 7.90 \\ \hline
    Sommerville No. 1 & 4 & 6 & 6 & 6 & 6 & 4 & 7.41 \\ \hline
    \multicolumn{8}{|c|}{Non-tiles} \\ \hline
     & $n_{12}$ & $n_{13}$ & $n_{14}$ & $n_{23}$ & $n_{24}$ & $n_{34}$ & Area \\ \hline
    NT(A) & 3 & 4 & 5 & 10 & 6 & 6 & 8.81 \\ \hline
    NT(B) & 3 & 5 & 5 & 10 & 10 & 4 & 8.81 \\ \hline
    NT(C) & 3 & 5 & 10 & 10 & 6 & 4 & 8.84 \\ \hline
    NT(D) & 3 & 6 & 10 & 10 & 10 & 3 & 9.28 \\ \hline
    NT(E) & 4 & 4 & 4 & 5 & 6 & 10 & 8.64 \\ \hline
    NT(F) & 4 & 5 & 6 & 5 & 6 & 5 & 7.53 \\ \hline
  \end{tabular}
  \label{tab:2pi/n}
\end{table}

First we check that the the dihedral angles of tetrahedra
in Table \ref{tab:2pi/n} form valid tetrahedra.
This is not trivial since there are no simple
sufficient  conditions on the dihedral angles.

\begin{lemma}
\label{lem:projratioedgelen}
In any tetrahedron, for any $\set{i,j,k,l}=\set{1,2,3,4}$, 
$$d_{kl}:d_{jl}:d_{jk}=
\abs{F_j}\sin\theta_{kl}:\abs{F_k}\sin\theta_{jl}:\abs{F_l}\sin\theta_{jk}.$$
\end{lemma}

\begin{proof}
If $h_i$ is the height corresponding to face $F_i$, then by projection
$$\abs{F_j}\sin\theta_{kl}=\frac{1}{2}h_id_{kl}.$$
Thus the ratio in the lemma statement holds by canceling out $h_i$.
\end{proof}

For some proofs we will need to compute the edge lengths of a tetrahedron given its dihedral angles. This computation is described in Lemma \ref{lem:lengthfromangle}.

\begin{lemma}
\label{lem:lengthfromangle}
Given the dihedral angles of a tetrahedron, the edge lengths can be computed
(up to scaling).
\end{lemma}

\begin{proof}
Lemma \ref{lem:dihedralrelation} can be used to compute
an area vector of the tetrahedron, up to scaling.
By Lemma \ref{lem:projratioedgelen}, for any $\set{i,j,k,l}=\set{1,2,3,4}$
the ratio
$$d_{kl}:d_{jl}:d_{jk}$$
can be computed.
From this and the areas of the faces, all edge lengths can be computed using simple properties of triangles.
\end{proof}

\begin{prop}
\label{prop:2pi/nvalid}
The dihedral angles specified in
Table \ref{tab:2pi/n} form valid tetrahedra.
\end{prop}

\begin{proof}
We check that the determinant condition in Lemma \ref{lem:dihedralrelation}
is satisfied symbolically using Mathematica.
This implies that the dihedral angles are exact.

To check that the dihedral angles form a valid tetrahedron,
we compute edge lengths as described in Lemma \ref{lem:lengthfromangle}
and use Lemma \ref{lem:edgecond} to verify that these edge lengths
form a tetrahedron.
Then Lemma \ref{lem:computedihedral}
allows us to compute the dihedral angles from edge lengths
to check that we get the original dihedral angles.
The last step is needed since invalid dihedral angles may give rise
to valid edge lengths.
The code for this computation is in the GitHub repository
in the folder \verb|2pi_over_n|.
\end{proof}

All tetrahedral tiles in Table \ref{tab:2pi/n} have been previously confirmed as tiles by Goldberg \cite{Gol}
(Sect. \ref{sec:prevtiles}).
The Sommerville tiles have the same name
as in the paper by Goldberg,
and the First Goldberg family refers to the first of the
three Goldberg infinite families.
Recall that by Proposition \ref{prop:goldberg12tile},
the First Goldberg family tiles face-to-face.

We now prove that Table \ref{tab:2pi/n} is exhaustive.
The approach is to first show that each dihedral angle cannot be very small,
and so the problem reduces to a finite search.
A computer program was written to perform the search, yielding
eleven candidate tetrahedra in Table \ref{tab:2pi/n}.
Finally, we show that exactly five of them tile space (face-to-face), as indicated in Table \ref{tab:2pi/n}.

\begin{lemma}
\label{lem:lessthan42}
If all dihedral angles of a tetrahedron are of the form
$\theta_{ij}=2\pi/n_{ij}$, then every $n_{ij}$ is less than 42.
\end{lemma}

\begin{proof}
Suppose to the contrary that $n_{12}\geq 42$, so that
$\theta_{12}\leq2\pi/42$. We will arrive at a contradiction using inequalities
on the dihedral angles.

By Lemma \ref{lem:dihedralineq},
$\theta_{13}+\theta_{14}+\theta_{23}+\theta_{24}<2\pi.$
Therefore either $\theta_{13}+\theta_{14}<\pi$ or $\theta_{23}+\theta_{24}<\pi$.
Assume that $\theta_{13}+\theta_{14}<\pi$.
The other inequality in Lemma \ref{lem:dihedralineq} implies
$$\frac{2\pi}{42}+\theta_{13}+\theta_{14}\geq \theta_{12}+\theta_{13}+\theta_{14}>\pi,$$
so that $\theta_{13}+\theta_{14}> 20\pi/21$. Hence
\begin{equation}
\label{eq:impossibleangle}
\frac{20\pi}{21}<\theta_{13}+\theta_{14}<\pi.
\end{equation}
We will show that this last inequality is impossible.

Notice that one of $\theta_{13}$ and $\theta_{14}$ has to be greater than
$10\pi/21$, which is greater than $2\pi/5$.
Assume that $\theta_{13}>2\pi/5$.
So $\theta_{13}$ is either $2\pi/3$ or $2\pi/4$.
If $\theta_{13}=2\pi/3$, then \eqref{eq:impossibleangle} reduces to
$$\frac{2\pi}{7}<\theta_{14}<\frac{2\pi}{6},$$
which is a contradiction. On the other hand, if $\theta_{13}=2\pi/4$,
then \eqref{eq:impossibleangle} reduces to
$$\frac{2\pi}{5}<\frac{19\pi}{42}<\theta_{14}<\frac{2\pi}{4},$$
which is also a contradiction.
We conclude that it is impossible to have $n_{12}\geq 42$.
\end{proof}

\begin{theorem}
\label{thm:2pi/nall}
The only tetrahedra with dihedral angles of the form $2\pi/n$, where $n\in\mb{N}$,
are those identified in Table \ref{tab:2pi/n},
up to permutation of vertices (reflection included).
\end{theorem}

\begin{proof}
By Lemma \ref{lem:lessthan42}, all dihedral angles are of the form
$2\pi/n$, where $3\leq n\leq 41$ is a natural number.
A computer search can be performed to find all sextuples
$(\theta_{12},\theta_{13},\theta_{14},\theta_{23},\theta_{24},\theta_{34})$
of numbers of this form which satisfy the determinant condition
in Lemma \ref{lem:dihedralrelation}. Moreover, the inequality
conditions in Lemma \ref{lem:dihedralineq} are used to narrow
down many superfluous possibilities. Finally, permutation of vertices
is taken into account to reduce thecut down search time and eliminate redundant results.
The final results are those in Table \ref{tab:2pi/n}.
The code for the computation is in the file \verb|2pi_over_n/search.m|.
\end{proof}

\begin{theorem}
\label{thm:2pi/ntile}
The only face-to-face tetrahedral tiles with dihedral angles of the form $2\pi/n$,
where $n\in\mb{N}$, are those indicated in Table \ref{tab:2pi/n},
up to permutation of vertices (reflection included).
\end{theorem}

\begin{proof}
By Theorem \ref{thm:2pi/nall}, we only need to show that the six so-called non-tiling tetrahedra
in Table \ref{tab:2pi/n} indeed do not tile face-to-face. The main tool used is Corollary \ref{cor:edge2pi/n}.

From the dihedral angles of Table \ref{tab:2pi/n}, we can compute the edge lengths using the method outlined in Lemma \ref{lem:lengthfromangle}. Note that this computation is numerical, so we can only confirm that edge lengths are distinct and not that they are identical. Therefore, we must consider separately the case where all edges that we claim have the same length do indeed have the same length and the case where they in fact do not have the same length.

Suppose first that all edges that we claim to have the same length do indeed have the same length. 
The proof for each of the non-tiling tetrahedra in Table \ref{tab:2pi/n} is given in Table \ref{tab:edgegraphs}. 
For each tetrahedron, we identify an edge $e_{ij}$ of length $d$ for which all edges of length $d$ have the same associated dihedral angle of the form $2\pi/n$ for $n$ odd and the $d$-edge-length graph has no closed walk of odd length. 
However, by Corollary \ref{cor:edge2pi/n}, if the tetrahedron tiled then the associated angle $\theta_{ij}$ would be of the form $\pi/n$.
Therefore, these tetrahedra do not tile.

\begin{table}
  \caption{Edge-length graphs of all non-tiles identified in Table \ref{tab:2pi/n}. All the graphs have no closed walk of odd length, so by Corollary \ref{cor:edge2pi/n}, $\theta_{ij}$ should be $\pi/n$ for some $n$, which is not the case; hence these tetrahedra cannot tile.}
  \begin{center}
  \begin{tabular}{| l | p{2.5cm} | l | m{6cm} |}
    \hline
    Name &
    Symmetries &
    Angle $\theta_{ij}$ &
    $d$-Edge-Length Graph \\ \hline
    
    NT(A) &
    None &
    $\theta_{12}=2\pi/3$ &
    \pathlentwo{(d_{13},d_{23})}{(d_{14},d_{24})}{12}{.75}{1} \\ \hline
    
    NT(B) &
    $d_{13}=d_{14}\newline d_{23}=d_{24}$ &
    $\theta_{13}=\theta_{14}=2\pi/5$ &
    \pathlentwo{(d_{12},d_{23})}{(d_{14},d_{34})}{13=14}{1.5}{1} \\ \hline
    
    NT(C) &
    $d_{13}=d_{14}$ &
    $\theta_{12}=2\pi/3$ &
    \pathlentwo{(d_{13},d_{23})}{(d_{14},d_{24})}{12}{.75}{1} \\ \hline
    
    NT(D) &
    $d_{12}=d_{34}\newline d_{14}=d_{23}$ &
    $\theta_{12}=\theta_{34}=2\pi/3$ &
    \pathlentwo{(d_{13},d_{23})}{(d_{14},d_{24})}{12=34}{1.5}{1} \\ \hline
    
    NT(E) &
    None &
    $\theta_{23}=2\pi/5$ &
    \pathlentwo{(d_{12},d_{13})}{(d_{24},d_{34})}{23}{.75}{1} \\ \hline
    
    NT(F) &
    $d_{13}=d_{23}\newline d_{14}=d_{24}$ &
    $\theta_{13}=\theta_{23}=2\pi/5$ &
    \pathlentwo{(d_{12},d_{23})}{(d_{14},d_{34})}{13=23}{1.5}{1} \\ \hline
  \end{tabular}
  \end{center}
  \label{tab:edgegraphs}
\end{table}

Now, suppose that some edges that we claim to have the same length in fact have distinct lengths. 
We can easily check that even if some edges that we claim to have the same length in fact have distinct lengths,
the edge-length graphs will still have no closed walk of odd length.
Then by Corollary \ref{cor:edge2pi/n}, if the tetrahedron tiled then the associated angle $\theta_{ij}$ would be of the form $\pi/n$, so these tetrahedra do not tile.
\end{proof}




\section{Classification of Tetrahedra}
\label{sec:classification}
In Section \ref{sec:classification}, we classify all tetrahedra
by distinct edge lengths.
The tetrahedra are classified into 25 types, up to permutation of vertices,
as listed explicitly in Table \ref{tab:tetrahedratypes}.
Different letters signify different edge lengths.
The summary of these 25 types, categorized according to
the number of distinct edge lengths and the number of congruent faces and sorted into several groups, is given in Figure \ref{fig:tetrahedrontable}.

After writing up this classification into 25 types and sending a copy to Wirth, he kindly sent us his 2013 version \cite{WirNew}.

First, it is shown that all 25 types exist.

\begin{prop}
All 25 types of tetrahedra presented in Table \ref{tab:tetrahedratypes} occur.
\end{prop}

\begin{proof}
By Lemma \ref{lem:edgecond}, any sextuple of edge lengths satisfying
certain inequalities forms a tetrahedron. Then the edge lengths
of a regular tetrahedron
can be adjusted slightly to make the tetrahedron belong to any of the 25 types,
and the edge lengths will still satisfy the inequalities.
\end{proof}

\begin{prop}
\label{prop:25typescomplete}
The 25 types of tetrahedra presented in Table \ref{tab:tetrahedratypes} are complete.
\end{prop}

\begin{proof}
Our proof is the same as Wirth and Dreiding's \cite[Sect. 1]{WirNew}.
We consider the number of distinct edge lengths and how many times each length occurs,
which Wirth and Dreiding call the \emph{lengths partition}. 
For each lengths partition, several arrangements of edge lengths into tetrahedra can be found; 
it is easy to check that these arrangements are the only ones possible and are unique up to rotation, translation, scaling, and reflection, resulting in the complete list of 25 total types of tetrahedra presented in Table \ref{tab:tetrahedratypes}.
\end{proof}

\begin{prop}
\label{prop:OPsommerville14}
The Sommerville Nos. 1--4 are the only tiles of the ten types in the \emph{orientation-preserving group} of Figure \ref{fig:tetrahedrontable}.
\end{prop}

\begin{proof}
It is easy to check that the ten types of the orientation-preserving group are their own reflections. 
Therefore by Definition \ref{def:orientation-preserving} these ten types can only tile in an orientation-preserving manner. 
An incomplete proof by Sommerville \cite{Solist}, later completed by
Edmonds \cite{Ed}, shows that the Sommerville No. 1--4 are the only orientation-preserving tiles.
\end{proof}

\begin{figure}[p!]
\includegraphics[angle=90, width = .9\textwidth]{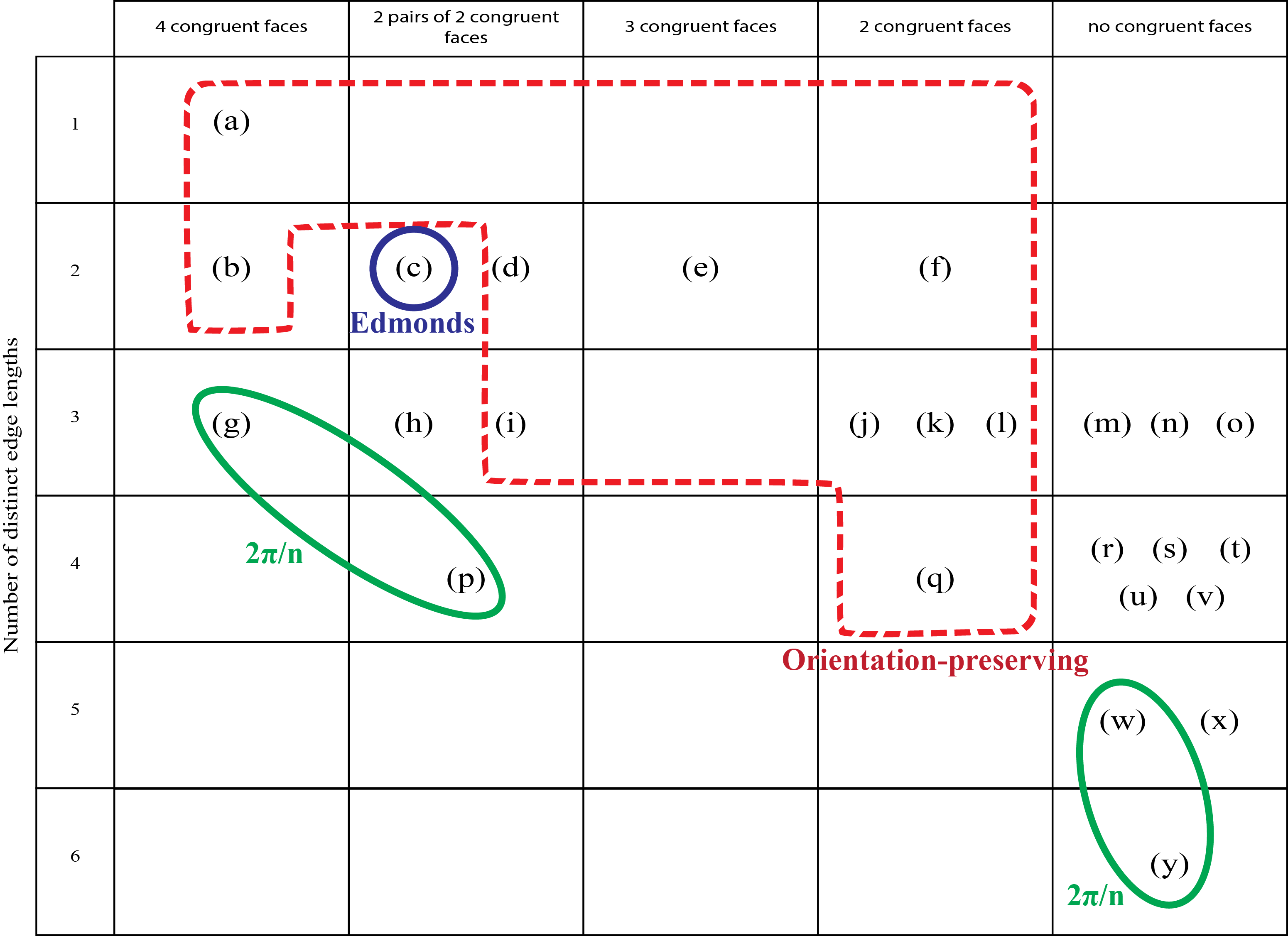}
\caption{Classification of all tetrahedra into 25 types, sorted by number of distinct edge lengths and number of congruent faces. The only tiles in the \emph{orientation-preserving group} are the Sommerville Nos. 1--4 (Prop. \ref{prop:OPsommerville14}). The \emph{Edmonds group} contains one tile (Prop. \ref{prop:somxi_tile}). The \emph{$2\pi/n$ group} contains no tiles (Thm. \ref{thm:2pingp_notile}). The rest, referred to as non-characterized, may have incomplete lists of tiles and are considered in Sections \ref{sec:10types}--\ref{sec:remainingcases}.}
\label{fig:tetrahedrontable}
\end{figure}
\clearpage

\newcolumntype{L}[1]{>{\raggedright\arraybackslash}m{#1}}
\newcolumntype{C}[1]{>{\centering\arraybackslash}m{#1}}

\begin{longtable}{| C{1cm} | C{4cm} | L{3cm} | L{6cm} |}	
\caption{Classification of tetrahedra into 25 types.}
    \label{tab:tetrahedratypes} \\ \hline
    Type 
    & Picture 
    & Group 
    & Tiles 
    \\ \hline
    \endfirsthead 
    
    \hline
    Type 
    & Picture 
    & Group 
    & Tiles 
    \\ \hline \endhead
    
    (a) 
    & \tetrahedron{a}{a}{a}{a}{a}{a}{1.15} 
    & Orientation-preserving 
    & Does not tile.
    \\ \hline
    
    (b) 
    & \tetrahedron{a}{b}{a}{a}{b}{a}{1.15}
    & Orientation-preserving 
    & Tiles only if Sommerville No. 1 ($a/b=\sqrt{3}/2$). 
    \\ \hline
    
    (c) 
    & \tetrahedron{a}{a}{b}{b}{b}{a}{1.15} 
    & Edmonds 
    & Tiles only if $a/b=\sqrt{2/3}$ or $\sqrt{3/2}$ by Edmonds \cite{Ed}.
    \\ \hline
    
    (d) 
    & \tetrahedron{a}{a}{a}{a}{a}{b}{1.15} 
    & Orientation-preserving 
    & Does not tile.
    \\ \hline
    
    (e) 
    & \tetrahedron{a}{a}{a}{b}{b}{b}{1.15} 
    & Orientation-preserving 
    & Does not tile. 
    \\ \hline
    
    (f) 
    & \tetrahedron{b}{a}{b}{a}{a}{a}{1.15} 
    & Orientation-preserving 
    & Does not tile. 
    \\ \hline
    
    (g) 
    & \tetrahedron{a}{b}{c}{c}{b}{a}{1.15} 
    & $2\pi/n$ 
    & Does not tile. 
    \\ \hline
    
    (h) 
    & \tetrahedron{a}{b}{c}{c}{b}{b}{1.15} 
    & Non-characterized
    & Tiles if in first Goldberg family ($c^2=b^2+a^2/3$).
    \\ \hline
    
    (i) 
    & \tetrahedron{a}{b}{a}{a}{c}{a}{1.15} 
    & Orientation-preserving 
    & Does not tile.
    \\ \hline
    
    (j) 
    & \tetrahedron{a}{a}{a}{b}{c}{b}{1.15} 
    & Orientation-preserving 
    & Tiles only if Sommerville No. 3 ($a:b:c=\sqrt{3}:2:2\sqrt{2}$)
    or No. 4 ($a:b:c=\sqrt{5}/2:\sqrt{3}:2$).
    \\ \hline
    
    (k) 
    & \tetrahedron{b}{c}{b}{a}{a}{a}{1.15} 
    & Orientation-preserving 
    & Does not tile.
    \\ \hline
    
    (l) 
    & \tetrahedron{a}{b}{b}{c}{c}{a}{1.15} 
    & Orientation-preserving 
    & Does not tile.
    \\ \hline
    
    (m) 
    & \tetrahedron{a}{b}{c}{a}{a}{a}{1.15} 
    & Non-characterized 
    & 
    \\ \hline
    
    (n) 
    & \tetrahedron{a}{b}{a}{a}{c}{b}{1.15} 
    & Non-characterized
    & 
    \\ \hline
    
    (o) 
    & \tetrahedron{a}{b}{c}{a}{c}{b}{1.15}
    & Non-characterized
    & 
    \\ \hline
    
    (p) 
    & \tetrahedron{a}{c}{b}{b}{d}{a}{1.15} 
    & $2\pi/n$ 
    & Does not tile.
    \\ \hline
    
    (q) 
    & \tetrahedron{a}{c}{b}{a}{d}{b}{1.15}
    & Orientation-preserving 
    & Tiles only if Sommerville No. 2 ($a:b:c:d=\sqrt{3}:\sqrt{2}:2:1$).
    \\ \hline
    
    (r) 
    & \tetrahedron{a}{a}{a}{b}{c}{d}{1.15} 
    & Non-characterized
    &
    \\ \hline
    
    (s) 
    & \tetrahedron{a}{b}{c}{d}{d}{d}{1.15} 
    & Non-characterized
    & 
    \\ \hline
    
    (t) 
    & \tetrahedron{a}{b}{a}{a}{c}{d}{1.15} 
    & Non-characterized
    & 
    \\ \hline
    
    (u) 
    & \tetrahedron{a}{a}{b}{b}{c}{d}{1.15} 
    & Non-characterized
    & 
    \\ \hline
    
    (v) 
    & \tetrahedron{a}{b}{c}{a}{c}{d}{1.15} 
    & Non-characterized 
    & Tiles if in second Goldberg family
    ($a^2 = c^2 + d^2/3$, $b^2 = a^2+ 4d^2/3$),
    including Baumgartner T2 ($a:b:c:d=\sqrt{3}:2:\sqrt{11}/2:\sqrt{3}/2$). 
    \\ \hline
    
    (w) 
    & \tetrahedron{b}{a}{c}{d}{a}{e}{1.15} 
    & $2\pi/n$ 
    & Does not tile.
    \\ \hline
    
    (x) 
    & \tetrahedron{a}{b}{c}{a}{d}{e}{1.15} 
    & Non-characterized
    & 
    \\ \hline
    
    (y) 
    & \tetrahedron{a}{b}{c}{d}{e}{f}{1.15} 
    & $2\pi/n$ 
    & Does not tile.
    \\ \hline
\end{longtable}

We now show that tetrahedra of the $2\pi/n$ group---types (g), (p), (w), and (y)---do not tile (face-to-face).
We first need the following lemmas describing the symmetries
of tetrahedra of types (g) and (p).

\begin{lemma}
\label{lem:typegsymmetry}
For a tetrahedron of type (g), the opposite dihedral angles
are equal: $\theta_{12}=\theta_{34}$,
$\theta_{13}=\theta_{24}$, and $\theta_{14}=\theta_{23}$.
\end{lemma}

\begin{proof}
Permute the vertices so that vertices $V_2$ and $V_3$ are switched
and vertices $V_1$ and $V_4$ are switched.
The resulting tetrahedron is identical to the original.
Hence $\theta_{12}=\theta_{34}$.
The other equalities follow from analogous reasoning.
\end{proof}

\begin{lemma}
\label{lem:typepsymmetry}
For a tetrahedron of type (p),
$\theta_{12}=\theta_{34}$ and $\theta_{14}=\theta_{23}$.
\end{lemma}

\begin{proof}
Permute the vertices so that vertices $V_1$ and $V_3$ are switched
and vertices $V_2$ and $V_4$ are switched.
The resulting tetrahedron is identical to the original one.
Therefore $\theta_{12}=\theta_{34}$ and $\theta_{14}=\theta_{23}$.
\end{proof}

The following theorem shows that tetrahedra in the $2\pi/n$ group, types (g), (p), (w), and (y), do not tile.

\begin{theorem}
\label{thm:2pingp_notile}
A tetrahedron of type (g), (p), (w) or (y) does not tile (face-to-face).
\end{theorem}

\begin{proof}
First, we show that all dihedral angles of a tetrahedral tile
of these types must be of the form $2\pi/n$.
For types (g), (p), and (w), the edge-length graph is used to prove this result.

\begin{longtable}{| C{1cm} | L{5cm} |}
\caption{Edge-length graphs of tetrahedra of types (g), (p), and (w).
For edge lengths that do not appear in this table,
there is only one edge of that length, and the corresponding
edge-length graph consists of two nodes connected by an edge.}
\label{tab:2pi/nedgelengraphs} \\ 
\hline 
Type &
Edge-Length Graph(s)
\\ \hline \endfirsthead

\hline
Type &
Edge-Length Graph(s)
\\ \hline \endhead

(g) &
\newline For edge length $a$:
\pathlentwo{(b,c)}{(c,b)}{12=34}{1.25}{1}
\newline \newline
For edge length $b$:
\pathlentwo{(a,c)}{(c,a)}{13=24}{1.25}{1}
\newline \newline
For edge length $c$:
\pathlentwo{(a,b)}{(b,a)}{14=23}{1.25}{1}
\\ \hline

(p) &
\newline For edge length $a$:
\pathlentwo{(b,c)}{(d,b)}{12=34}{1.25}{1}
\pathlentwo{(b,d)}{(c,b)}{12=34}{1.25}{1}
\newline \newline
For edge length $b$:
\pathlentwo{(a,c)}{(d,a)}{14=23}{1.25}{1}
\pathlentwo{(a,d)}{(c,a)}{14=23}{1.25}{1}
\\ \hline

(w) &
\newline For edge length $a$:
\pathlentwo{(b,d)}{(c,e)}{13}{.75}{1}
\pathlentwo{(d,b)}{(e,c)}{13}{.75}{1}
\pathlentwo{(b,c)}{(d,e)}{24}{.75}{1}
\pathlentwo{(c,b)}{(e,d)}{24}{.75}{1}
\\ \hline
\end{longtable}

\begin{itemize}
\item Type (g): 
By Lemma \ref{lem:typegsymmetry},
$\theta_{12}=\theta_{34}$,
$\theta_{13}=\theta_{24}$, and $\theta_{14}=\theta_{23}$.
By Corollary \ref{cor:edge2pi/n}, every dihedral angle is of the form $2\pi/n$. 

\item Type (p):
By Lemma \ref{lem:typepsymmetry},
$\theta_{12}=\theta_{34}$ and $\theta_{14}=\theta_{23}$.
By Corollary \ref{cor:edge2pi/n}, every dihedral angle is of the form $2\pi/n$. 

\item Type (w):
From the $a$-edge-length graph, any closed walk in the graph may only pass through
one of the labels 13 or 24, but not both. That is, 
the two edges of common length cannot coincide in a tiling.
Then the all dihedral angles of the tetrahedron must be of the form $2\pi/n$.

\item Type (y):
Every edge has a distinct length, so every dihedral angle
is of the form $2\pi/n$ by Corollary \ref{cor:edge2pi/n}.
\end{itemize}

By Theorem \ref{thm:2pi/ntile}, a tetrahedral tile of these four types must be in Table \ref{tab:2pi/n}.
By inspection, tetrahedral tiles in Table \ref{tab:2pi/n}
are not of these four types;
hence, a tetrahedron of any of these four types does not tile.
\end{proof}

Section \ref{sec:classification} is summarized by the following theorem, which states that the Sommerville No. 1 minimizes surface area among the tiles of all completely characterized types.

\begin{theorem}
\label{thm:15types_sono1best}
The Sommerville No. 1 uniquely minimizes surface area among all tiles of the orientation-preserving, Edmonds, and $2\pi/n$ groups of Figure \ref{fig:tetrahedrontable}.
\end{theorem}

\begin{proof}
By Propositions \ref{prop:somxi_tile} and \ref{prop:OPsommerville14} and Theorem \ref{thm:2pingp_notile}, the tiles of the orientation-preserving, Edmonds, and $2\pi/n$ groups of Figure \ref{fig:tetrahedrontable} are among those identified in previous literature. Then by Corollary \ref{cor:prevknown_so1best}, the Sommerville No. 1 uniquely minimizes surface area among the tiles of these groups.
\end{proof}

It remains to be shown that the Sommerville No. 1 minimizes surface area among the non-characterized ten types of Figure \ref{fig:tetrahedrontable}, which are examined in Sections \ref{sec:10types}--\ref{sec:remainingcases}.

\section{Previously Non-Characterized Ten Types}
\label{sec:10types}
In Section \ref{sec:10types}, we deduce some necessary conditions on the ten non-characterized types defined in Section \ref{sec:classification}
in Figure \ref{fig:tetrahedrontable}.
We will use these
and the isoperimetric conditions of Section \ref{sec:boundiso}
to eliminate isoperimetric candidates of these ten types in
Section \ref{sec:code}.

First, we set up
linear systems on dihedral angles of the 10 non-characterized types
based on their edge-length graphs (Table \ref{tab:10typesedgelengraphs}).

\begin{prop}
The 10 non-characterized types have the edge-length graphs
and linear combinations of dihedral angles as indicated in
Table \ref{tab:10typesedgelengraphs}.
\end{prop}

\begin{proof}
The edge-length graphs are constructed as described in Section \ref{sec:edgelengraphs}.
We can deduce the linear combinations by considering
closed walks in the graph using Proposition \ref{prop:edgegraphlincomb}.
Parities of the coefficients follow from Remark \ref{rem:closedwalk}.
For type (h), equalities of dihedral angles follow from the same argument
as in Lemma \ref{lem:typegsymmetry}.
\end{proof}

\begin{longtable}{| C{1cm} | L{7cm} | L{8cm} |}
\caption{Edge-length graphs for the (previously) non-characterized types, whose tiling properties are not entirely understood.
For edge lengths that do not appear in this table,
there is only one edge of that length, and the corresponding
edge-length graph has no closed walk of odd length.}
\label{tab:10typesedgelengraphs} \\ \hline
Type &
Edge-Length Graph(s) &
Linear Systems
\\ \hline \endfirsthead

\hline
Type &
Edge-Length Graph(s) &
Linear Systems
\\ \hline \endhead
 
(h) &
\newline For edge length $b$:
\pathlenfour{(a,c)}{(c,b)}{(b,c)}{(c,a)}{13=24}{34}{13=24}{1}{.7}{.4}{1}{1}
\newline \newline
For edge length $c$:
\pathlenthree{(a,b)}{(b,b)}{(b,a)}{14=23}{14=23}{1.4}{1.3}{1}
& $\theta_{12}$ is of the form $\pi/n$.
\newline
\newline $n_{13}\theta_{13}+n_{34}\theta_{34}=2\pi$.
\newline $\theta_{13}=\theta_{24}$.
\newline $n_{13}$ and $n_{34}$ are even.
\newline
\newline $\theta_{14}=\theta_{23}$ is of the form $\pi/n$. 
\\ \hline

(m) &
\newline For edge length $a$:
\begin{tikzpicture}[scale=1,transform shape]
\draw [black] (0,0) node [anchor=north] {$(a,a)$}
-- node [anchor=west] {23}
(0,1.25) node [anchor=south] {$(a,b)$};
\draw [black] (-.5,-.1)
-- node [anchor=north east] {34}
(-2, 1.25) node [anchor=south] {$(b,c)$};
\draw [black] (.6,1.65)
-- node [anchor=south] {12} (1.35,1.65);
\draw [black] (.5,-.1)
-- node [anchor=north west] {24}
(2,1.25) node[anchor=south] {$(a,c)$};
\draw [black] (0,-.75)
-- node [anchor=west] {23}
(0,-2) node[anchor=north] {$(b,a)$};
\draw [black] (-.5,-.65)
-- node [anchor=south east] {24}
(-2, -2) node [anchor=north] {$(c,a)$};
\draw [black] (-1.375,-2.5)
-- node [anchor=north] {12} (-.625,-2.5);
\draw [black] (.5,-.65)
-- node [anchor=south west] {34}
(2,-2) node [anchor=north] {$(c,b)$};
\end{tikzpicture}
& $\theta_{13}$ and $\theta_{14}$ are of the form $\pi/n$.
\newline
\newline $n_{12}\theta_{12} + n_{23}\theta_{23} + n_{24}\theta_{24} + n_{34}\theta_{34} = 2\pi$.
\newline $n_{12}, n_{23},$ and $n_{24}$ have the same parity. $n_{34}$ is even.
\\ \hline
 
(n) &
\newline For edge length $a$:
\begin{tikzpicture}[scale=1,transform shape]
\draw [black] (0,0) node [anchor=east] {$(b,c)$}
-- node [anchor=south] {23}
(1,0) node [anchor=west] {$(b,a)$};
\draw [black] (2.3,0)
-- node [anchor=south] {12}
(3.25,0) node [anchor=west] {$(a,c)$};
\draw [black] (3.9,-.35)
-- node [anchor=west] {14}
(3.9,-1.35) node [anchor=north] {$(b,b)$};
\draw [black] (3.9,-2)
-- node [anchor=west] {14}
(3.9,-3) node [anchor=north] {$(c,a)$};
\draw [black] (3.25,-3.5)
-- node [anchor=south] {12}
(2.4,-3.5) node [anchor=east] {$(a,b)$};
\draw [black] (1,-3.5)
-- node [anchor=south] {23}
(0,-3.5) node [anchor=east] {$(c,b)$};
\end{tikzpicture}
\newline \newline
For edge length $b$:
\pathlenfive{(a,c)}{(b,a)}{(a,a)}{(a,b)}{(c,a)}{34}{13}{13}{34}{1.4}{.75}{.7}{.4}{.75}{1}
& $\theta_{24}$ is of the form $\pi/n$.
\newline
\newline $n_{12}\theta_{12} + n_{14}\theta_{14} + n_{23}\theta_{23}= 2\pi$.
\newline $n_{12}, n_{14}$, and $n_{23}$ are even.
\newline
\newline $n_{13}\theta_{13} + n_{34}\theta_{34} = 2\pi$.
\newline $n_{13}$ and $n_{34}$ are even.
\\ \hline
 
(o) &
\newline For edge length $a$:
\pathlenfive{(b,c)}{(b,a)}{(c,c)}{(a,b)}{(c,b)}{23}{12}{12}{23}{1.4}{.75}{.7}{.4}{.75}{1}
\newline \newline For edge length $b$:
\pathlenfive{(a,c)}{(b,c)}{(a,a)}{(c,b)}{(c,a)}{34}{13}{13}{34}{1.4}{.75}{.7}{.4}{.75}{1}
\newline \newline For edge length $c$:
\pathlenfive{(a,b)}{(a,c)}{(b,b)}{(c,a)}{(b,a)}{24}{14}{14}{24}{1.4}{.75}{.7}{.4}{.75}{1}
 & $n_{12}\theta_{12} + n_{23}\theta_{23} = 2\pi$.
\newline $n_{12}$ and $n_{23}$ are even.
\newline
\newline $n_{13}\theta_{13} + n_{34}\theta_{34} = 2\pi$.
\newline $n_{13}$ and $n_{34}$ are even.
\newline
\newline $n_{14}\theta_{14} + n_{24}\theta_{24} = 2\pi$.
\newline $n_{14}$ and $n_{24}$ are even.
\\ \hline
 
(r) &
\newline For edge length $a$:
\trianglegraph{(a,c)}{(a,b)}{(a,d)}{12}{13}{14}{.5}{.4}{.6}{.75}{.3}{.15}{1}
\trianglegraph{(c,a)}{(b,a)}{(d,a)}{12}{13}{14}{.5}{.4}{.6}{.75}{.3}{.15}{1}
& $\theta_{23}$, $\theta_{24}$, and $\theta_{34}$ are of the form $\pi/n$.
\newline
\newline $n_{12}\theta_{12} + n_{13}\theta_{13} + n_{14}\theta_{14}= 2\pi$.
\newline $n_{12}, n_{13},$ and $n_{14}$ have the same parity.
\\ \hline
 
(s) &
\newline For edge length $d$:
\begin{tikzpicture}[scale=1,transform shape]
\draw [black] (0,0) node[anchor=north] {$(d,d)$}
-- node [anchor=west] {24}
(0,.75) node [anchor=south] {$(a,c)$};
\draw [black] (-.5,-.1)
-- (-1.5, .75) node [anchor=south] {$(a,b)$};
\node [anchor=east] at (-1,.2) {23};
\draw [black] (.5,-.1)
-- (1.5,.75) node [anchor=south] {$(b,c)$};
\node [anchor=west] at (1,.2) {34};
\draw [black] (0,-.75)
-- node [anchor=west] {24}
(0,-1.5) node [anchor=north] {$(c,a)$};
\draw [black] (-.5,-.65)
-- (-1.5, -1.5) node [anchor=north] {$(c,b)$};
\node [anchor=east] at (-1,-.9) {34};
\draw [black] (.5,-.65)
-- (1.5,-1.5) node [anchor=north] {$(b,a)$};
\node [anchor=west] at (1,-.9) {23};
\end{tikzpicture}
 & $\theta_{12}$, $\theta_{13}$, and $\theta_{14}$ are of the form $\pi/n$.
\newline
\newline $n_{23}\theta_{23} + n_{24}\theta_{24} + n_{34}\theta_{34}= 2\pi$.
\newline $n_{23}, n_{24},$ and $n_{34}$ are even.
\\ \hline
 
(t) & \newline
For edge length $a$:
\pathlenfour{(d,c)}{(b,a)}{(a,c)}{(b,d)}{23}{12}{14}{.75}{.7}{.4}{.75}{1}
\pathlenfour{(c,d)}{(a,b)}{(c,a)}{(d,b)}{23}{12}{14}{.75}{.7}{.4}{.75}{1}
& $\theta_{13}$, $\theta_{24}$, and $\theta_{34}$ are of the form $\pi/n$.
\newline
\newline $n_{12}\theta_{12} + n_{14}\theta_{14} + n_{23}\theta_{23}= 2\pi$.
\newline $n_{12}, n_{14},$ and $n_{23}$ are even.
\\ \hline
 
(u) & \newline
For edge length $a$:
\pathlenthree{(b,c)}{(a,b)}{(b,d)}{12}{13}{1.4}{.75}{1}
\pathlenthree{(c,b)}{(b,a)}{(d,b)}{12}{13}{1.4}{.75}{1}
\newline \newline
For edge length $b$:
\pathlentwo{(a,c)}{(a,d)}{14}{.75}{1}
\pathlentwo{(c,a)}{(d,a)}{14}{.75}{1}
\pathlenthree{(c,d)}{(a,a)}{(d,c)}{23}{23}{1.4}{.75}{1}
& $\theta_{24}$ and $\theta_{34}$ are of the form $\pi/n$.
\newline
\newline $n_{12}\theta_{12} + n_{13}\theta_{13}= 2\pi$.
\newline $n_{12}$ and $n_{13}$ are even.
\newline
\newline $\theta_{14}$ and $\theta_{23}$
are of the form $\pi/n$.
\\ \hline
 
(v) &
\newline For edge length $a$:
\pathlenfive{(c,d)}{(a,b)}{(c,c)}{(b,a)}{(d,c)}{23}{12}{12}{23}{1.4}{.75}{.7}{.4}{.75}{1}
\newline \newline
For edge length $c$:
\pathlenthree{(b,d)}{(a,c)}{(a,d)}{14}{24}{1.4}{.75}{1}
\pathlenthree{(d,b)}{(c,a)}{(d,a)}{14}{24}{1.4}{.75}{1}
& $\theta_{13}$ and $\theta_{34}$ are of the form $\pi/n$.
\newline
\newline $n_{12}\theta_{12} + n_{23}\theta_{23} = 2\pi$.
\newline $n_{12}$ and $n_{23}$ are even.
\newline
\newline $n_{14}\theta_{14} + n_{24}\theta_{24} = 2\pi$.
\newline $n_{14}$ and $n_{24}$ are even.
\\ \hline
 
(x) & \newline
For edge length $a$:
\pathlenthree{(d,e)}{(a,b)}{(d,c)}{23}{12}{1.4}{.75}{1}
\pathlenthree{(e,d)}{(b,a)}{(c,d)}{23}{12}{1.4}{.75}{1}
& $\theta_{13}$, $\theta_{14}$, $\theta_{24}$, and $\theta_{34}$ are of the form $\pi/n$.
\newline
\newline $n_{12}\theta_{12} + n_{23}\theta_{23}= 2\pi$.
\newline $n_{12}$ and $n_{23}$ are even.
\\ \hline
\end{longtable}

The vanishing determinant condition on the dihedral angles
given in Lemma \ref{lem:dihedralrelation} can be strengthened
given that some edge lengths are equal.
The strengthened condition for nine of the ten previously non-characterized types is presented in Table \ref{tab:10typesdetcond}.
Type (u) is omitted because we will not use this condition in the code in Section \ref{sec:code}; 
in the code, it is treated as a subcase of type (x).

First we show how equality of edge lengths can be used to derive
relationships between face areas and dihedral angles.

\begin{lemma}
\label{lem:equaledgetofaceratio}
In a tetrahedron, for $\set{i,j,k,l}=\set{1,2,3,4}$, if $d_{ij}=d_{ik}$, then
$$\abs{F_k}\sin\theta_{ij}=\abs{F_j}\sin\theta_{ik}.$$
\end{lemma}

\begin{proof}
By Lemma \ref{lem:projratioedgelen},
$$d_{ij}:d_{ik}=\abs{F_k}\sin\theta_{ij}:\abs{F_j}\sin\theta_{ik}.$$
So the desired equality holds given that $d_{ij}=d_{ik}$.
\end{proof}

\begin{prop}
\label{prop:matrixcorrect}
The 10 non-characterized types have the face area relationships 
as given in Table \ref{tab:10typesdetcond},
and the matrix for each type
has nontrivial null space.
\end{prop}

\begin{proof}
The face area relationships can be derived by
applying Lemma \ref{lem:equaledgetofaceratio} to
adjacent edges of equal lengths specified by the types.
For type (h), the fact that $\abs{F_1}=\abs{F_2}$
and $\abs{F_3}=\abs{F_4}$ follows from direct inspection
of edge lengths.
Finally, because the vector $(\abs{F_1}, \abs{F_2}, \abs{F_3}, \abs{F_4})^T$
is in the null space of the matrix in Lemma \ref{lem:dihedralrelation},
these face area relationships imply that the matrices in the table have nontrivial null space.
\end{proof}

\begin{longtable}{| C{1cm} | L{5.6cm} | L{9cm} |}
\caption{Strengthened determinant conditions for the (previously) non-characterized types, whose tiling properties are not entirely understood.
The $C_i$ refers to column $i$ of the matrix in Lemma \ref{lem:dihedralrelation}. The matrices for the strengthened determinant condition all have nontrivial null space.}
\label{tab:10typesdetcond} \\ \hline
Type &
Face Area Relationships &
Strengthened Determinant Condition
\\ \hline \endfirsthead

\hline
Type &
Face Area Relationships &
Strengthened Determinant Condition
\\ \hline \endhead
 
(h) &
$\abs{F_1}=\abs{F_2}$,
\newline $\abs{F_3}=\abs{F_4}$,
\newline $\abs{F_1}\sin\theta_{34}=\abs{F_4}\sin\theta_{13}$.
&
\newline
$\begin{pmatrix}
\vline \\
\sin\theta_{13}(C_1+C_2) + \sin\theta_{34}(C_3+C_4) \\
\vline
\end{pmatrix}$
\\ \hline

(m) &
$\abs{F_1}\sin\theta_{23}=\abs{F_3}\sin\theta_{12}$,
\newline $\abs{F_2}\sin\theta_{34}=\abs{F_3}\sin\theta_{24}
=\abs{F_4}\sin\theta_{23}$.
&
\newline
$\begin{pmatrix}
\vline \\
\sin\theta_{12}\sin\theta_{34}C_1 + \sin\theta_{23}\sin\theta_{24}C_2 \\
+ \sin\theta_{23}\sin\theta_{34}C_3 + \sin\theta_{24}\sin\theta_{34}C_4 \\
\vline
\end{pmatrix}$
\\ \hline
 
(n) &
$\abs{F_1}\sin\theta_{23}=\abs{F_3}\sin\theta_{12}$,
$\abs{F_2}\sin\theta_{14}=\abs{F_4}\sin\theta_{12}$,
$\abs{F_1}\sin\theta_{34}=\abs{F_4}\sin\theta_{13}$.
&
\newline
$\begin{pmatrix}
\vline \\
\sin\theta_{12}\sin\theta_{13}\sin\theta_{14}C_1 \\
+ \sin^2\theta_{12}\sin\theta_{34}C_2 \\
+ \sin\theta_{13}\sin\theta_{14}\sin\theta_{23}C_3 \\
+ \sin\theta_{12}\sin\theta_{14}\sin\theta_{34}C_4 \\
\vline
\end{pmatrix}$
\\ \hline
 
(o) &
$\abs{F_1}\sin\theta_{23}=\abs{F_3}\sin\theta_{12}$,
$\abs{F_1}\sin\theta_{34}=\abs{F_4}\sin\theta_{13}$,
$\abs{F_1}\sin\theta_{24}=\abs{F_2}\sin\theta_{14}$.
&
\newline
$\begin{pmatrix}
\vline \\
\sin\theta_{12}\sin\theta_{13}\sin\theta_{14}C_1 \\
+ \sin\theta_{12}\sin\theta_{13}\sin\theta_{24}C_2 \\
+ \sin\theta_{23}\sin\theta_{13}\sin\theta_{14}C_3 \\
+ \sin\theta_{12}\sin\theta_{34}\sin\theta_{14}C_4 \\
\vline
\end{pmatrix}$
\\ \hline
 
(r) &
$\abs{F_2}\sin\theta_{13}=\abs{F_3}\sin\theta_{12}$,
$\abs{F_3}\sin\theta_{14}=\abs{F_4}\sin\theta_{13}$.
&
\newline
$\begin{pmatrix}
\vline & \vline \\
C_1 & \sin\theta_{12}C_2+\sin\theta_{13}C_3+\sin\theta_{14}C_4 \\
\vline & \vline
\end{pmatrix}$
\\ \hline
 
(s) &
$\abs{F_2}\sin\theta_{34}=\abs{F_3}\sin\theta_{24}
=\abs{F_4}\sin\theta_{23}$.
&
\newline
$\begin{pmatrix}
\vline & \vline \\
\vline & \sin\theta_{23}\sin\theta_{24}C_2 \\
C_1 &+ \sin\theta_{23}\sin\theta_{34}C_3 \\
\vline & +\sin\theta_{24}\sin\theta_{34}C_4 \\
\vline & \vline
\end{pmatrix}$
\\ \hline
 
(t) &
$\abs{F_1}\sin\theta_{23}=\abs{F_3}\sin\theta_{12}$,
$\abs{F_2}\sin\theta_{14}=\abs{F_4}\sin\theta_{12}$.
&
\newline
$\begin{pmatrix}
\vline & \vline \\
\sin\theta_{12}C_1 & \sin\theta_{12}C_2 \\
+\sin\theta_{23}C_3 & +\sin\theta_{14}C_4 \\
\vline & \vline
\end{pmatrix}$
\\ \hline
 
(u) & -
& -
\\ \hline
 
(v) &
$\abs{F_1}\sin\theta_{23}=\abs{F_3}\sin\theta_{12}$,
$\abs{F_1}\sin\theta_{24}=\abs{F_2}\sin\theta_{14}$.
&
\newline
$\begin{pmatrix}
\vline & \vline \\
\sin\theta_{12}\sin\theta_{14}C_1 & \vline \\
+\sin\theta_{12}\sin\theta_{24}C_2 & C_4 \\
+\sin\theta_{23}\sin\theta_{14}C_3 & \vline \\
\vline & \vline
\end{pmatrix}$
\\ \hline
 
(x) &
$\abs{F_1}\sin\theta_{23}=\abs{F_3}\sin\theta_{12}$.
&
\newline
$\begin{pmatrix}
\vline & \vline & \vline \\
\sin\theta_{12}C_1 + \sin\theta_{23}C_3 & C_2 & C_4 \\
\vline & \vline & \vline
\end{pmatrix}$
\\ \hline
\end{longtable}

\smallskip

\section{Isoperimetry}
\label{sec:boundiso}
In Section \ref{sec:boundiso}, we establish bounds on
the dihedral angles of a tetrahedron which can beat the Sommerville No. 1, which will be used in code in Section \ref{sec:code}.
First we need a lemma relating the volume of a tetrahedron
to the areas of its faces.

\begin{lemma}
\label{lem:volformula}
Let $V$ be the volume of a tetrahedron.
For $\set{i,j,k,l}=\set{1,2,3,4}$, if we let
$\theta_{ijk}$ to be the angle $V_iV_jV_k$, then
$$V^2=\frac{2}{9}\abs{F_j}\abs{F_k}\abs{F_l}
\sin\theta_{ij}\sin\theta_{ik}\sin\theta_{jik}.$$
\end{lemma}

\begin{proof}
Let $h_i$ be the height of the tetrahedron from vertex $V_i$
to face $F_i$. Moreover, let $h_{i,jk}$ be the height from vertex $V_i$
to edge $e_{jk}$ in the triangle $V_iV_jV_k$.
Notice that
$$V=\frac{1}{3}h_j\abs{F_j}=\frac{1}{3}h_k\abs{F_k}.$$
Hence
$$V^2=\frac{1}{9}h_jh_k\abs{F_j}\abs{F_k}.$$
It remains to show that
$$h_jh_k=2\abs{F_l}\sin\theta_{ij}\sin\theta_{ik}\sin\theta_{jik}.$$
\noindent By projection, we have
$$h_j=h_{j,ik}\sin\theta_{ik}=d_{ij}\sin\theta_{jik}\sin\theta_{ik}$$
and similarly
$$h_k=d_{ik}\sin\theta_{jik}\sin\theta_{ij}.$$
So it remains to prove
$$\abs{F_l}=\frac{1}{2}d_{ij}d_{ik}\sin\theta_{jik},$$
which is true.
\end{proof}

The following proposition gives bounds on dihedral angles
of a tetrahedron in terms of its normalized surface area.

\begin{prop}
\label{prop:lowerbddihedral}
Suppose that a tetrahedron has volume $V$ and surface area $S$.
Then all its dihedral angles $\theta_{ij}$ must satisfy the inequality
$$\sin\theta_{ij}\geq \frac{243}{S^3/V^2}.$$
\end{prop}

\begin{proof}
We can assume $(i,j)=(1,2)$.
By Lemma \ref{lem:volformula}, we can bound
$$V^2 \leq \frac{2}{9}\abs{F_2}\abs{F_3}\abs{F_4}\sin\theta_{12}.$$
Similarly, we have
$$V^2 \leq \frac{2}{9}\abs{F_1}\abs{F_3}\abs{F_4}\sin\theta_{12}.$$
So we may conclude
$$V^2 \leq \frac{2}{9}\paren{\frac{\abs{F_1}+\abs{F_2}}{2}}
\abs{F_3}\abs{F_4}\sin\theta_{12}.$$
The arithmetic mean-geometric mean inequality implies that
$$\paren{\abs{F_1}+\abs{F_2}}\abs{F_3}\abs{F_4}
\leq \paren{\frac{\abs{F_1}+\abs{F_2}+\abs{F_3}+\abs{F_4}}{3}}^3
=\frac{S^3}{27}.$$
Using this last inequality, we can get the desired bound.
\end{proof}

We can now apply this result to tetrahedra
with less surface area than the Sommerville No. 1.

\begin{cor}
\label{cor:dihedralbdsommerville}
A tetrahedron of unit volume with surface area less
than or equal to the Sommerville No. 1 has all dihedral angles
in the range $[\theta_0,\pi-\theta_0]$, where
$$\theta_0=\sin^{-1} \paren{\frac{27}{32\sqrt{2}}} \approx 36.63^\circ.$$
\end{cor}

\begin{proof}
The Sommerville No. 1 has normalized surface area
$$\frac{S}{V^{2/3}}=2^{11/6}\cdot 3^{2/3}.$$
The desired bound follows from 
Proposition \ref{prop:lowerbddihedral} and direct computation.
\end{proof}

\section{Code}
\label{sec:code}

In Section \ref{sec:code}, we start by using information from Sections \ref{sec:10types} and \ref{sec:boundiso} to
reduce all potential candidates to beat the Sommerville No. 1 tetrahedron to a finite number of cases. There are \numcasesinitial\ cases
in total, and they are enumerated by computer code. Then we use the computer to eliminate all but \numcasesafterelim\ of them. Those remaining \numcasesafterelim\ cases will be handled by special arguments
in Section \ref{sec:remainingcases}.

\subsection*{Overview}
For each edge-length graph for a (previously) non-characterized type, for the dihedral angles, Table \ref{tab:10typesedgelengraphs} of Section \ref{sec:10types} provides linear equations with nonnegative integer coefficients. By Corollary \ref{cor:dihedralbdsommerville},
for a tetrahedron to beat the Sommerville No. 1, its dihedral angles must all be greater than 36.5 degrees, so at most 9 of them can stack around a common edge. This means that the sum of the coefficients in each equation is at most 9, and so there are a finite number of possible sets of coefficients. Each set of coefficients is one case.
We use the parities of coefficients provided in
Table \ref{tab:10typesedgelengraphs} and the symmetries of each type to reduce the number of cases.
Finally, we make use of reductions from one type to another
to assume that certain coefficients must be nonzero.
This is important not only for reducing the number of cases
but also because the degenerate cases where many coefficients
are zero are in general more difficult to eliminate.
The full details of this setup are discussed in Proposition \ref{prop:codesetupcorrect}.

Each case consists of a number of equations
whose variables are the dihedral angles.
Tables \ref{tab:10typesedgelengraphs} and \ref{tab:10typesdetcond}
give enough information so that in every case
the number of equations is greater than the number of variables
by exactly one.
Specifically, we start with a determinant condition on the dihedral
angles. Each time we make two edges have equal length,
their associated dihedral angles are related to one another
by a linear equation in Table \ref{tab:10typesedgelengraphs},
which produces one more variable. The equal-edge-lengths condition implies face-area relationships in Table \ref{tab:10typesdetcond}, yielding one more equation. Thus the number of equations remains one more than the number of variables, and we expect that most of the cases will produce no solution.
This is indeed the case for all but \numcasesaftereintmin\ cases.
To prove this rigorously, we use the interval minimization method
provided by Stan Wagon \cite[Chapt. 4]{BLWW}
to show that in the relevant domain of variables,
all equations cannot be simultaneously satisfied.
The details are given in Proposition \ref{prop:intervalminimize}.

For the \numcasesaftereintmin\ remaining cases that survive
interval minimization, we attempt to use the computer algebra system
of Mathematica to directly solve the system of equations.
Those which produce no solution or whose solutions conflict
with known conditions are discarded.
The remaining \numcasesafterelim\ cases which cannot be eliminated by
this method are considered by hand and finally eliminated
in Section \ref{sec:remainingcases}.
The details are given in Proposition \ref{prop:eliminateafterintmin}.

\subsection*{Implementation}
The code considers 9 types of tetrahedra which roughly
represent the 10 non-characterized types but are not exactly the same, because one code type can correspond to many of the original
tetrahedra types and vice versa. We will use the phrase ``code types'' to refer to the 9 types in the code and the phrase ``original types'' to refer to the 10 non-characterized types
from Section \ref{sec:classification}.
By correspondence we mean that each case encountered in a certain
code type can come from any of the corresponding original types.

Table \ref{tab:codetypecorr} maps all such correspondences.
Each of the 10 non-characterized types except (u), 
which is treated as a subcase of (x), 
has one main corresponding code type.
The main code type covers most of the tetrahedra in that original type, except for possible degenerate cases, which are instead covered in the non-main corresponding code types.

\begin{table}
\caption{The correspondences between the nine code types
and ten non-characterized types of tetrahedra.
``Main Corr. Type'' refers to the main non-characterized type
corresponding to each code type. ``Other Corr. Types'' refers to
other non-characterized types corresponding to each code type.}
\label{tab:codetypecorr}
\begin{tabular}{| c | l | l |}
\hline
Code Type & Main Corr. Type & Other Corr. Types \\ \hline
 \texttt{aaabcd} & (r) & -- \\
 \texttt{abaacb} & (n) & -- \\
 \texttt{abaacd} & (t) & (n) \\
 \texttt{abcaaa} & (m) & -- \\
 \texttt{abcacb} & (o) & -- \\
 \texttt{abcacd} & (v) & (o) \\
 \texttt{abcade} & (x) & (u), (n), (o), (v) \\
 \texttt{abccbb} & (h) & -- \\
 \texttt{abcddd} & (s) & -- \\
\hline
\end{tabular}
\end{table}

We first describe the setup of each code type presented in Table \ref{tab:codetypeassumption}.

\begin{lemma}
\label{lem:typersolid}
A face-to-face tetrahedral tile of type (r) has the solid angle at vertex $V_1$,
$\Omega_1$, of the form $4\pi/n$ for some $n\in\mb{N}$.
\end{lemma}

\begin{proof}
For a tetrahedron of type (r), vertex $V_1$ cannot meet with
other vertices in the tiling. Hence the solid angle $\Omega_1$
must divide $4\pi$.
\end{proof}

\begin{prop}
\label{prop:codesetupcorrect}
The nine code types with assumptions and symmetries given in
Table \ref{tab:codetypeassumption} cover all potential candidates to beat or tie with the Sommerville No. 1 of the ten non-characterized types, as in Table \ref{tab:codetypecorr}.
\end{prop}

\begin{proof}
We argue that for each non-characterized type, the corresponding code types
given in Table \ref{tab:codetypecorr} are sufficient to cover all potential candidates of that type.
We first focus on the main corresponding code type. Then we consider degenerate cases
which correspond to the non-main code types. Finally, we argue that the symmetries
are correct.

For the original types (r), (n), (t), (m), (o), (v), (x), (h), and (s), the main corresponding
code type is given in Table \ref{tab:codetypecorr}. The assumptions for these code types
in Table \ref{tab:codetypeassumption}, except the ones that say certain $n_{ij}$ are nonzero,
are justified as follows.
\begin{itemize}
\item For conditions of the form $\theta_{ij}\in\set{\pi/2,\pi/3,\pi/4}$,
Table \ref{tab:10typesedgelengraphs} says that $\theta_{ij}$ is of the form $\pi/n$
in the Linear Systems column. By Corollary \ref{cor:dihedralbdsommerville}, $n$ can only be
2, 3 or 4.
\item That $\theta_{ij},\theta_{jk},\theta_{ik}$ cannot all be $\pi/2$
is true because otherwise the faces $F_i$, $F_j$, and $F_k$ must be perpendicular to
the face $F_l$, and the tetrahedron degenerates into an infinite triangular prism.
\item That $\theta_{ij}+\theta_{ik}+\theta_{il}>\pi$ is true
because $\Omega_i=\theta_{ij}+\theta_{ik}+\theta_{il}-\pi$ is the solid angle at
vertex $V_i$ (Lemma \ref{lem:dihedralineq}).
\item That $\theta_{ik} + \theta_{il} + \theta_{jl} + \theta_{jk} < 2\pi$ is true by Lemma \ref{lem:dihedralineq}.
\item That the matrix of each original type given in Table \ref{tab:10typesdetcond}
has nontrivial null space is argued in Proposition \ref{prop:matrixcorrect}.
\item The linear systems $\sum n_{ij}\theta_{ij}=2\pi$ and the parities of $n_{ij}$
are given in Table \ref{tab:10typesedgelengraphs}.
By Corollary \ref{cor:dihedralbdsommerville}, the sum of these $n_{ij}$ is at most 9.
\item For code type \verb|aaabcd|, the fact that $\Omega_1$ divides $4\pi$
follows from the fact that this code type only corresponds to type (r) and Lemma \ref{lem:typersolid}.
\end{itemize}

Now we argue for the correctness of the conditions ``at least two of the $n_{ij}$ are nonzero,''
henceforth called the \emph{nonzero condition}.
By Proposition \ref{prop:edgegraphlincomb}, for any linear combination
$$\sum n_{ij}\theta_{ij}=2\pi$$
in Table \ref{tab:codetypeassumption}, we can assume any coefficient $n_{ij}$ is nonzero
(but perhaps not two coefficients at once). So we can assume \emph{at least two} coefficients
are nonzero unless all angles $\theta_{ij}$ in the linear combination are of the form $2\pi/n$.
This forms the basis of the arguments described in Table \ref{tab:codereduction}.

The arguments in Table \ref{tab:codereduction} read as follows.
For each linear system (referred to by the condition number), if the nonzero condition
of that system does not hold, then certain dihedral angles must be of the form $2\pi/n$,
and that possibility is covered by another code type indicated in the table.
For example, for type (n), if only the nonzero condition for (3) does not hold,
then $\theta_{12}$, $\theta_{14}$ and $\theta_{23}$ are of the form $2\pi/n$.
By inspection, this possbility can be covered by the code type \verb|abacde|.
The phrase ``Reduce to the $2\pi/n$ case'' means that for that possibility,
all dihedral angles are of the form $2\pi/n$, so by Theorem \ref{thm:2pi/ntile},
the tetrahedron does not beat or tie with the Sommerville No. 1.
For type (u), refer to the information in Table \ref{tab:10typesedgelengraphs}.
Hence all cases where the nonzero conditions fail are covered by other code types
or Theorem \ref{thm:2pi/ntile}, so that the nonzero conditions can be enforced.

\begin{longtable}{| C{1.5cm} | C{2.5cm} | L{12cm} |}
\caption{Reductions of the 10 non-characterized types to non-main corresponding code types.}
\label{tab:codereduction} \\ \hline
Type &
Main Corr. Code Type &
Reductions
\\ \hline \endfirsthead

\hline
Type &
Main Corr. Code Type &
Reductions
\\ \hline \endhead
 
(r) & \texttt{aaabcd} &
(3): Reduce to the $2\pi/n$ case.
\\ \hline

(n) & \texttt{abaacb} &
(3): Reduce to code type \verb|abcade| (with some permutation).
\newline (4): Reduce to code type \verb|abaacd|.
\newline (3) and (4): Reduce to the $2\pi/n$ case.
\\ \hline

(t) & \texttt{abaacd} &
(3): Reduce to the $2\pi/n$ case.
\\ \hline

(m) & \texttt{abcaaa} &
(3): Reduce to the $2\pi/n$ case.
\\ \hline

(o) & \texttt{abcacb} &
Note the cyclic symmetry between (2), (3), and (4).
\newline (3): Reduce to code type \verb|abcacd|.
\newline (3) and (4): Reduce to code type \verb|abcade|.
\newline (2), (3), and (4): Reduce to the $2\pi/n$ case.
\\ \hline

(v) & \texttt{abcacd} &
(3): Reduce to code type \verb|abcade| (with some permutation).
\newline (4): Reduce to code type \verb|abcade|.
\newline (3) and (4): Reduce to the $2\pi/n$ case.
\\ \hline

(x) & \texttt{abcade} &
(3): Reduce to the $2\pi/n$ case.
\\ \hline

(h) & \texttt{abccbb} &
(4): Reduce to the $2\pi/n$ case.
\\ \hline

(s) & \texttt{abcddd} &
(3): Reduce to the $2\pi/n$ case.
\\ \hline

(u) & -- &
If $n_{12}$ and $n_{13}$ are nonzero, reduce to code type \verb|abcade|
(with some permutation).
\newline Otherwise, reduce to the $2\pi/n$ case.
\\ \hline
\end{longtable}

Finally, the inequalities in the Symmetries column can be enforced
due to the symmetries of the assumptions of each code type, which can be easily checked.
Therefore the nine code types cover all potential candidates to beat or tie with
the Sommerville No. 1 of the ten non-characterized types.
\end{proof}

\begin{longtable}{| C{2.5cm} | L{8cm} | L{6cm} |}
\caption{Assumptions of each code type.
Matrix of each type refers to the matrix in the strengthened
determinant condition in Table \ref{tab:10typesdetcond}.}
\label{tab:codetypeassumption} \\ \hline
Code Type &
Assumptions &
Symmetries
\\ \hline \endfirsthead

\hline
Code Type &
Assumptions &
Symmetries
\\ \hline \endhead
 
\texttt{aaabcd} &
\begin{enumerate}[leftmargin=*]
\item $\theta_{23},\theta_{24},\theta_{34} \in \set{\pi/2,\pi/3,\pi/4}$.
\newline $\theta_{23},\theta_{24},\theta_{34}$ cannot all be $\pi/2$.
\item Matrix of type (r) has nontrivial null space ($d_{12}=d_{13}=d_{14}$).
\item $n_{12}\theta_{12}+n_{13}\theta_{13}+n_{14}\theta_{14}=2\pi$.
\newline $n_{12}+n_{13}+n_{14} \leq 9$.
\newline $n_{12}$, $n_{13}$, and $n_{14}$ have the same parity and at least two are nonzero.
\item $\Omega_1$ divides $4\pi$.
\end{enumerate} &
$n_{12} \leq n_{13} \leq n_{14}$.
\newline If $n_{12}=n_{13}$, then $\theta_{24}\geq\theta_{34}$.
\newline If $n_{13}=n_{14}$, then $\theta_{23}\geq\theta_{24}$.
\\ \hline

\texttt{abaacb} &
\begin{enumerate}[leftmargin=*]
\item $\theta_{24} \in \set{\pi/2,\pi/3,\pi/4}$.
\item Matrix of type (n) has nontrivial null space ($d_{12}=d_{14}=d_{23}$, $d_{13}=d_{34}$).
\item $n_{12}\theta_{12}+n_{14}\theta_{14}+n_{23}\theta_{23}=2\pi$.
\newline $n_{12}+n_{14}+n_{23} \leq 9$.
\newline $n_{12}$, $n_{14}$, and $n_{23}$ are even and at least two are nonzero.
\item $n_{13}\theta_{13}+n_{34}\theta_{34}=2\pi$.
\newline $n_{13}+n_{34}\leq 9$.
\newline $n_{13}$ and $n_{34}$ are even and nonzero.
\end{enumerate} &
--
\\ \hline

\texttt{abaacd} &
\begin{enumerate}[leftmargin=*]
\item $\theta_{13},\theta_{24},\theta_{34} \in \set{\pi/2,\pi/3,\pi/4}$.
\item Matrix of type (t) has nontrivial null space ($d_{12}=d_{14}=d_{23}$).
\item $n_{12}\theta_{12}+n_{14}\theta_{14}+n_{23}\theta_{23}=2\pi$.
\newline $n_{12}+n_{14}+n_{23} \leq 9$.
\newline $n_{12}$, $n_{14}$, and $n_{23}$ are even and at least two are nonzero.
\end{enumerate} &
$n_{14}\leq n_{23}$.
\newline If $n_{14}=n_{23}$, then $\theta_{13}\geq\theta_{24}$.
\\ \hline

\texttt{abcaaa} &
\begin{enumerate}[leftmargin=*]
\item $\theta_{13},\theta_{14} \in \set{\pi/2,\pi/3,\pi/4}$.
\item Matrix of type (m) has nontrivial null space ($d_{12}=d_{23}=d_{24}=d_{34}$).
\item $n_{12}\theta_{12}+n_{23}\theta_{23}+n_{24}\theta_{24}+n_{34}\theta_{34}=2\pi$.
\newline $n_{12}+n_{23}+n_{24}+n_{34} \leq 9$.
\newline $n_{12}$, $n_{23}$, and $n_{24}$ have the same parity,
and $n_{34}$ is even.
\newline At least two of $n_{12},n_{23},n_{24},n_{34}$ are nonzero.
\end{enumerate} &
$n_{23}\leq n_{24}$.
\newline If $n_{23}=n_{24}$, then $\theta_{13}\geq\theta_{14}$.
\\ \hline

\texttt{abcacb} &
\begin{enumerate}[leftmargin=*]
\item Matrix of type (o) has nontrivial null space ($d_{12}=d_{23}$, $d_{13}=d_{34}$,
$d_{14}=d_{24}$).
\item $n_{12}\theta_{12}+n_{23}\theta_{23}=2\pi$.
\newline $n_{12}+n_{23} \leq 9$.
\newline $n_{12}$ and $n_{23}$ are even and nonzero.
\item $n_{13}\theta_{13}+n_{34}\theta_{34}=2\pi$.
\newline $n_{13}+n_{34} \leq 9$.
\newline $n_{13}$ and $n_{34}$ are even and nonzero.
\item $n_{14}\theta_{14}+n_{24}\theta_{24}=2\pi$.
\newline $n_{14}+n_{24} \leq 9$.
\newline $n_{14}$ and $n_{24}$ are even and nonzero.
\end{enumerate} &
$(n_{12},n_{23})\leq (n_{13},n_{34})$
and $(n_{12},n_{23})\leq (n_{14},n_{24})$
in the dictionary order.
\\ \hline

\texttt{abcacd} &
\begin{enumerate}[leftmargin=*]
\item $\theta_{13},\theta_{34} \in \set{\pi/2,\pi/3,\pi/4}$.
\item Matrix of type (v) has nontrivial null space ($d_{12}=d_{23}$, $d_{14}=d_{24}$).
\item $n_{12}\theta_{12}+n_{23}\theta_{23}=2\pi$.
\newline $n_{12}+n_{23} \leq 9$.
\newline $n_{12}$ and $n_{23}$ are even and nonzero.
\item $n_{14}\theta_{14}+n_{24}\theta_{24}=2\pi$.
\newline $n_{14}+n_{24} \leq 9$.
\newline $n_{14}$ and $n_{24}$ are even and nonzero.
\end{enumerate} &
--
\\ \hline

\texttt{abcade} &
\begin{enumerate}[leftmargin=*]
\item $\theta_{13},\theta_{14},\theta_{24},\theta_{34} \in \set{\pi/2,\pi/3,\pi/4}$.
\newline $\theta_{13},\theta_{14},\theta_{34}$ cannot all be $\pi/2$.
\newline $\theta_{14}+\theta_{24}+\theta_{34}>\pi$.
\item Matrix of type (x) has nontrivial null space ($d_{12}=d_{23}$).
\item $n_{12}\theta_{12}+n_{23}\theta_{23}=2\pi$.
\newline $n_{12}+n_{23} \leq 9$.
\newline $n_{12}$ and $n_{23}$ are even and nonzero.
\end{enumerate} &
$n_{12}\leq n_{23}$.
\newline If $n_{12}=n_{23}$, then $\theta_{14}\geq \theta_{34}$.
\\ \hline

\texttt{abccbb} &
\begin{enumerate}[leftmargin=*]
\item $\theta_{12},\theta_{14}=\theta_{23} \in \set{\pi/2,\pi/3,\pi/4}$.
\item Matrix of type (h) has nontrivial null space ($d_{14}=d_{23}$,
$d_{13}=d_{24}=d_{34}$).
\item $\theta_{13}=\theta_{24}$.
\item $n_{13}\theta_{13}+n_{34}\theta_{34}=2\pi$.
\newline $n_{13}+n_{34} \leq 9$.
\newline $n_{13}$ and $n_{34}$ are even and nonzero.
\end{enumerate} &
--
\\ \hline

\texttt{abcddd} &
\begin{enumerate}[leftmargin=*]
\item $\theta_{12},\theta_{13},\theta_{14} \in \set{\pi/2,\pi/3,\pi/4}$.
\newline $\theta_{12}+\theta_{13}+\theta_{14}>\pi$.
\item Matrix of type (s) has nontrivial null space ($d_{23}=d_{24}=d_{34}$).
\item $n_{23}\theta_{23}+n_{24}\theta_{24}+n_{34}\theta_{34}=2\pi$.
\newline $n_{23}+n_{24}+n_{34} \leq 9$.
\newline $n_{23}$, $n_{24}$, and $n_{34}$ are even and at least two are nonzero.
\end{enumerate} &
$n_{23} \leq n_{24} \leq n_{34}$.
\newline If $n_{23}=n_{24}$, then $\theta_{13}\geq\theta_{14}$.
\newline If $n_{24}=n_{34}$, then $\theta_{12}\geq\theta_{13}$.
\\ \hline
\end{longtable}

We now describe how the cases are eliminated. First, from the setup, there are a total of
\numcasesinitial\ cases. Then, using interval minimization technique, this can be reduced
to \numcasesaftereintmin\ cases (Prop. \ref{prop:intervalminimize}).
Finally, by solving the system of equations directly, this can be further reduced to
\numcasesafterelim\ cases (Prop. \ref{prop:eliminateafterintmin}).

\begin{prop}
\label{prop:intervalminimize}
From the \numcasesinitial\ cases of the nine code types in Table \ref{tab:codetypeassumption},
all but \numcasesaftereintmin\ cases cannot occur.
\end{prop}

\begin{proof}
Each choice of $\theta_{ij}\in\set{\pi/2,\pi/3,\pi/4}$ and coefficients $n_{ij}$
in Table \ref{tab:codetypeassumption} gives one case. There are a total of \numcasesinitial\ cases
once enumerated. For each case, we explain how to set up variables and equations.
Then we explain how to determine that the system of equations has no solution.

From a linear system $\sum n_{ij}\theta_{ij}=2\pi$ consisting of $k$ terms, one angle $\theta_{ij}$
can be written in terms of $k-1$ other angles, provided that the coefficient $n_{ij}$ is nonzero.
Sometimes one coefficient is nonzero across all cases of a certain type, but sometimes
different coefficients have to be chosen. Once this is done, the linear system
will give rise to $k-1$ variables. We do this for all systems of that code type to get all variables.

The equations come from the condition that the matrix of each type has nontrivial null space in Table \ref{tab:codetypeassumption}.
If the matrix is $4\times n$, there are $5-n$ equations, each saying that an $n\times n$ determinant vanishes.
The number of equations is always greater than the number of variables by one,
so the system is overdetermined and we expect that most cases will produce no solution.
The choice of which determinants to use
affects the number of remaining cases, which for our computation is \numcasesaftereintmin.
To get this number, we put in some effort to choose determinants which give few remaining cases.

To determine whether determinants $d_1,d_2,\dots,d_m$ can vanish simultaneously, we minimize
$$d_1^2+d_2^2+\dots+d_m^2$$
with respect to all variables satisfying
\begin{enumerate}
\item inequalities on sums of dihedral angles given in Lemma \ref{lem:dihedralineq}, and
\item bounds on the dihedral angles given in Corollary \ref{cor:dihedralbdsommerville}.
\end{enumerate}
For the bounds in Corollary \ref{cor:dihedralbdsommerville}, we use $36.5^\circ$
as a numerical lower bound of $\theta_0$. We do not want to get very close to the exact value
to prevent round-off errors. On the other hand, it is undesirable to use
values less than or equal to $36^\circ$
because this gives superfluous cases containing dihedral angles $2\pi/5$.

All files in this computation are in the folder \verb|remaining_types| in the GitHub repository.
We ran this computation first in the files \verb|Type_xxxxxx.nb|, where \verb|xxxxxx| is the code type.
However, this uses Mathematica's \verb|NMinimize| function to minimize, which does not prove that the minimum is greater than zero.
To prove this rigorously, we utilize Stan Wagon's \cite[Chapt. 4]{BLWW}
interval minimization code in the file \verb|IntervalMinimize.nb|.
However, the interval minimization code can only minimize with respect to a
hyperrectangular region. Hence we make a bounding box of the region of variables that
satisfies conditions (1) and (2) above. This produces one superflouous case
not present in the \verb|NMinimize| approach, which is of the type
\verb|abcddd|. Fortunately, this case can be eliminated just by inspection.
The details are in the file \verb|Type_abcddd_prove.nb|.

The files that use interval minimization to rigorously eliminate all but
\numcasesaftereintmin\ cases are \verb|Type_xxxxxx_prove.nb|.
The summary of this elimination process is in \verb|casesummary.txt|.
The work of all the code described here proves the proposition.
\end{proof}

\begin{prop}
\label{prop:eliminateafterintmin}
From the \numcasesaftereintmin\ cases remaining after interval minimization
of Proposition \ref{prop:intervalminimize}, all but
\numcasesafterelim\ cases, which will be dealt with in
Section \ref{sec:remainingcases}, cannot match or improve upon the Sommerville No. 1.
\end{prop}

\begin{proof}
From the \numcasesaftereintmin\ remaining cases, we attempt to solve the
system of equations directly (symbolically). Then each case can be eliminated
if the result is one of the following:
\begin{enumerate}
\item No solution.
\item Solved to a closed-form solution with dihedral angles all of the form $2\pi/n$.
\item For code type \verb|aaabcd|, solved to a closed-form solution for which
the solid angle $\Omega_1$ does not divide $4\pi$.
\end{enumerate}
That possibility (1) can be eliminated is obvious. For possibility (2), by Theorem
\ref{thm:2pi/ntile}, no tetrahedron in this case can beat or tie with the Sommerville No. 1.
For possibility (3), by Lemma \ref{lem:typersolid}, the tetrahedra in this case do not tile.
A case cannot be eliminated if the result is one of the following:
\begin{enumerate}
\item Cannot solve (either to a closed-form solution or to no solution).
\item Timeout after 5 minutes.
\end{enumerate}
The files that carry out this computation are \verb|Type_xxxxxx.nb|, where \verb|xxxxxx|
is the code type, in the folder \verb|remaining_types|.
The summary of this process is in \verb|casesummary.txt|.
After the process has been carried out, there remain \numcasesafterelim\ cases
as in Section \ref{sec:remainingcases}.
\end{proof}

Table \ref{tab:codenumbers} presents the number of cases left at each stage of the
elimination process for each code type.
We summarize the results of this section (Props. \ref{prop:codesetupcorrect},
\ref{prop:intervalminimize}, and \ref{prop:eliminateafterintmin}) in the following proposition.

\begin{table}
\caption{Summary of results from code. The main corresponding type to each code type is also given. ``Num. Cases'' gives the number of cases that fit initial restrictions (Prop. \ref{prop:codesetupcorrect}).
``After Int. Min.'' gives the number of cases remaining after interval minimization
(Prop. \ref{prop:intervalminimize}). ``After Eliminate'' gives the number of cases remaining after elimination based on logical restrictions (Prop. \ref{prop:eliminateafterintmin}). A dash ``--'' in the ``After Eliminate'' column indicates that the type was already down to 0 cases and elimination was not applied. This process is outlined in Section \ref{sec:code}. After coding, a total of \numcasesafterelim\ cases remain; these will be considered in Section \ref{sec:remainingcases}.}
\label{tab:codenumbers}
\begin{tabular}{| c | c | c | c | c |}
\hline
Code Type & Main Corr. Type & Num. Cases & After Int. Min.& After Eliminate \\ \hline
 \texttt{aaabcd} & (r) & 224 & 7 & 2 \\
 \texttt{abaacb} & (n) & 396 & 1 & 1\\
 \texttt{abaacd} & (t) & 315 & 8 & 2\\
 \texttt{abcaaa} & (m) & 459 & 0 & -- \\
 \texttt{abcacb} & (o) & 91 & 0 & -- \\
 \texttt{abcacd} & (v) & 324 & 6 & 1\\
 \texttt{abcade} & (x) & 144 & 5 & 0\\
 \texttt{abccbb} & (h) & 54 & 2 & 1\\
 \texttt{abcddd} & (s) & 67 & 0 & --\\
\hline
 \multicolumn{2}{|c|}{Total} & \numcasesinitial\ & \numcasesaftereintmin\ & \numcasesafterelim\ \\ \hline
\end{tabular}
\end{table}

\begin{prop}
\label{prop:allfromcode}
From the ten non-characterized types, only seven cases as dealt with in Section \ref{sec:remainingcases}
can potentially beat or tie with the Sommerville No. 1.
\end{prop}

\section{Remaining Cases}
\label{sec:remainingcases}

Section \ref{sec:remainingcases} examines individually the cases that remain after code is used to eliminate candidates to tile with less surface area than the Sommerville No. 1. The code used to arrive at these remaining cases is described in Section \ref{sec:code}, 
and a summary of the number of remaining cases in each code type is given by Table \ref{tab:codenumbers}.
We find that out of \numcasesafterelim\ remaining cases, none tile with less surface area than the Sommerville No. 1; therefore, the Sommerville No. 1 is the least-area face-to-face tetrahedral tile even among non-orientation-preserving tiles.

Notice that each code type in Table \ref{tab:codenumbers} may correspond to more than one
non-characterized type (Table \ref{tab:codetypecorr}). For example, for code type \verb|abaacd|,
the edge lengths $b$ and $d$ may not be different if the tetrahedron comes from type (n).
In general, edge lengths denoted by the same letter are equal, but edge lengths denoted by different letters may or may not be equal.

The remaining cases are not particularly special; they are artifacts of the way the code was implemented
(Prop. \ref{prop:intervalminimize}). Choosing different determinants to base the code upon may produce different cases and possibly different numbers of cases. Regardless, there will be no cases that both tile and have less surface area than the Sommerville No. 1.

\subsection*{Remaining case 1 (RC1)}
RC1 belongs to the code type \texttt{aaabcd}. It is characterized by
\begin{equation}
\label{eq:RC1}
\text{RC1: }\begin{cases}
& \theta_{23}=\pi/3,	\\
& \theta_{24}=\pi/2, \\
& \theta_{34}=\pi/4, \\
& \theta_{12} + \theta_{13} + 3\theta_{14} = 2\pi.
\end{cases}
\end{equation}

\begin{prop}
\label{prop:RC1}
RC1 does not tile.
\end{prop}

\begin{proof}
By Table \ref{tab:codetypecorr}, RC1 can only be of type (r).
By Lemma \ref{lem:typersolid}, if RC1 tiles, then the solid angle
$$\Omega_1 = \theta_{12} + \theta_{13} + \theta_{14} - \pi$$
(Lemma \ref{lem:dihedralineq}) must divide $4\pi$.
From interval minimization, we have obtained the interval
in which $\theta_{12}$ and $\theta_{13}$ must lie.
By \eqref{eq:RC1}, we can compute the interval in which
$\theta_{14}$ must lie. Then we can compute the interval
containing the solid angle $\Omega_1$ and find that
$$\frac{4\pi}{\Omega_1} \in \bracket{17.3205, 17.4558}.$$
So $\Omega_1$ does not divide $4\pi$, a contradiction.
The code for this computation is in the file
\verb|Type_aaabcd_remaining.nb|.
\end{proof}

\subsection*{Remaining case 2 (RC2)}
RC2 belongs to the code type \texttt{aaabcd}. It is characterized by
\begin{equation}
\label{eq:RC2}
\text{RC2: } \begin{cases}
& \theta_{23}=\pi/4,	\\
& \theta_{24}=\pi/2,	\\
& \theta_{34}=\pi/4, \\
& \theta_{12} + \theta_{13} + 3\theta_{14} = 2\pi.
\end{cases}
\end{equation}

\begin{prop}
\label{prop:RC2}
If RC2 tiles, it must be the Sommerville No. 3.
\end{prop}

\begin{proof}
We perform an analysis similar to that of RC1.
By Lemma \ref{lem:typersolid}, if RC1 tiles, then the solid angle
$\Omega_1$ must divide $4\pi$.
From interval minimization, we obtained the intervals in which
$\theta_{12}$ and $\theta_{13}$ must lie, which give
the following range for $\Omega_1$:
$$\frac{4\pi}{\Omega_1} \in \bracket{11.9772, 12.0241}.$$
Because the left-hand side must be an integer, it is exactly 12.
Substituting this back into \eqref{eq:RC2} yields
$$\theta_{14} = \frac{\pi}{3} \text{ and } \theta_{12} + \theta_{13} = \pi.$$
We can now solve for $\theta_{12}$ directly using
the determinant condition in Lemma \ref{lem:dihedralrelation}, yielding
$$(\theta_{12},\theta_{13},\theta_{14},\theta_{23},\theta_{24},\theta_{34})
= \paren{\frac{\pi}{3},\frac{2\pi}{3},\frac{\pi}{3},
\frac{\pi}{4},\frac{\pi}{2},\frac{\pi}{4}}.$$
Thus, by Table \ref{tab:2pi/n}, RC1 is the Sommerville No. 3. 
The code for this computation is in the file
\verb|Type_aaabcd_remaining.nb|.
\end{proof}

\subsection*{Remaining case 3 (RC3)}
RC3 belongs to the code type \texttt{abaacb}. It is characterized by
$$
\text{RC3: } \begin{cases}
& \theta_{24}=\pi/2,	\\
& \theta_{12} + \theta_{14} = \pi/2, \\
& \theta_{13} + \theta_{34} =\pi.
\end{cases}
$$

\begin{prop}
\label{prop:RC3}
RC3 has greater normalized surface area than the Sommerville No. 1.
\end{prop}

\begin{proof}
Mathematica can solve for dihedral angles of RC3 in closed form.
By direct computation of surface area $S$ and volume $V$, RC3 has
normalized surface area $S/V^{2/3}$ of approximately 8.395.
The normalized surface area of the Sommerville No. 1 is
approximately 7.413.
The code for this computation is in the file
\verb|Type_abaacb_remaining.nb|.
\end{proof}

%

\subsection*{Remaining case 4 (RC4)}
RC4 belongs to the code type \texttt{abaacd}. It is characterized by
\begin{equation}
\label{eq:RC4}
\text{RC4: } \begin{cases}
& \theta_{13}=\pi/2, \\
& \theta_{24}=\pi/2, \\
& \theta_{34}=\pi/3, \\
& \theta_{12} + 2\theta_{23} = \pi.
\end{cases}
\end{equation}

\begin{prop}
\label{prop:RC4}
RC4 belongs to Goldberg's first infinite family.
\end{prop}

\begin{proof}
Mathematica can solve for $\theta_{14}$ as
$$\theta_{14}=\frac{\pi-\theta_{12}}{2}.$$
By this and \eqref{eq:RC4}, the dihedral angles of RC4 are 
$$
(\theta_{12},\theta_{13},\theta_{14},\theta_{23},\theta_{24},\theta_{34})
= \paren{\pi-2\alpha, \frac{\pi}{2}, \alpha, \alpha, \frac{\pi}{2}, \frac{\pi}{3}}.
$$
By Proposition \ref{prop:dihedralfixes}, 
RC4 belongs to the first Goldberg family.
\end{proof}

\subsection*{Remaining case 5 (RC5)}
RC5 belongs to the code type \texttt{abaacd}. It is characterized by
\begin{equation}
\label{eq:RC5}
\text{RC5: } \begin{cases}
& \theta_{13}=\pi/2,	\\
& \theta_{24}=\pi/2,	\\
& \theta_{34}=\pi/3, \\
& \theta_{12} + \theta_{14} + \theta_{23} = \pi.
\end{cases}
\end{equation}

\begin{prop}
\label{prop:RC5}
RC5 belongs to Goldberg's first infinite family.
\end{prop}

\begin{proof}
As in the case of RC4, Mathematica can solve for $\theta_{14}$ as
$$\theta_{14}=\frac{\pi-\theta_{12}}{2}.$$
From (\ref{eq:RC5}),
the dihedral angles of RC5 are
$$
(\theta_{12},\theta_{13},\theta_{14},\theta_{23},\theta_{24},\theta_{34})
= \paren{\pi-2\alpha, \frac{\pi}{2}, \alpha, \alpha, \frac{\pi}{2}, \frac{\pi}{3}}.
$$
By Proposition \ref{prop:dihedralfixes}, RC5 belongs to the first Goldberg family.	
\end{proof}

\subsection*{Remaining case 6 (RC6)}
RC6 belongs to the code type \texttt{abcacd}. It is characterized by
\begin{equation}
\label{eq:RC6}
\text{RC6: } \begin{cases}
& \theta_{13}=\pi/2,	\\
& \theta_{34}=\pi/3, \\
& \theta_{12} + \theta_{23} = \pi/2, \\
& \theta_{14} + \theta_{24} = \pi.
\end{cases}
\end{equation}

\begin{prop}
\label{prop:RC6}
RC6 belongs to Goldberg's second infinite family.
\end{prop}

\begin{proof}
The idea of the proof is to show that two copies of RC6
form a tetrahedron in Goldberg's first infinite family.
Put two copies of RC6, $V_1V_2V_3V_4$ and $V_1'V_2'V_3'V_4'$,
together by rotating and attaching the faces opposite vertex $V_3$
as in Figure \ref{fig:RC6}.

\begin{figure}[ht]
\begin{tikzpicture}[scale=1.5,transform shape]
\draw [black] (0,3) node [anchor=south] {\scriptsize $V_2, V_1'$} -- (-2.6,0.6) node [anchor=east] {\scriptsize $V_3$};
\draw [black] (0,3) -- (-0.8,-1) node [anchor=north] {\scriptsize $V_1, V_2'$};
\draw [black] (0,3) -- (2.2,-0.2) node [anchor=west] {\scriptsize $V_3'$};
\draw [black] (-2.6,0.6) -- (-0.8,-1);
\draw [dashed, black] (-2.6,0.6) -- (2.2,-0.2);
\draw [black] (-0.8,-1) -- (2.2,-0.2);

\draw [dashed, black] (0,3) -- (0.2,0.1);
\draw [dashed, black] (-0.8,-1) -- (0.2,0.1);
\node [anchor=south east] at (0.2,0.1) {\scriptsize $V_4$};

\node [anchor=center] at (1.5,1.5) {\scriptsize $b, \pi/2$};
\node [anchor=center] at (1,-0.7) {\scriptsize $a, \alpha$};
\node [anchor=center] at (-0.1,-0.4) {\scriptsize $c$};
\node [anchor=center] at (0.2,1.4) {\scriptsize $c$};
\node [anchor=center] at (1.1,0.2) {\scriptsize $d, \pi/3$};
\node [anchor=center] at (-1.5,2) {\scriptsize $a, \alpha$};
\node [anchor=center] at (-2.1,-0.4) {\scriptsize $b, \pi/2$};
\node [anchor=center] at (-1.3,0.6) {\scriptsize $d, \pi/3$};
\node [anchor=center] at (1.5,2.9) {\scriptsize $a, \pi-2\alpha$};

\draw[->, thick] (-0.15, 2.2) to [out = 0, in = 270](1.4, 2.8);
\end{tikzpicture}
\caption{Tetrahedron constructed by rotating RC6 and attaching the faces opposite vertex $V_3$, labeled with the same vertices as RC6. Note that $\theta_{14} + \theta_{24}= \pi$, so $V_3'$ lies on the same plane $V_1,V_3,V_4$ and also on the same plane as $V_2,V_3,V_4$.}
\label{fig:RC6}
\end{figure}
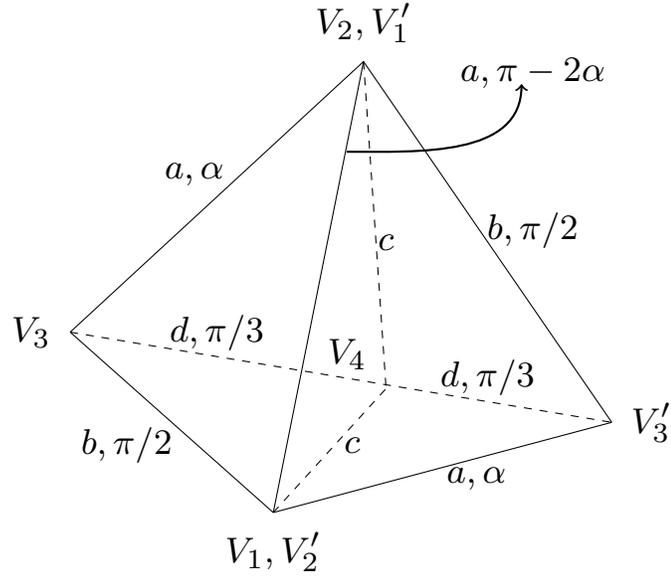

Because $\theta_{14} + \theta_{24}= \pi$, $V_3'$ lies on the same plane
as $V_1,V_3,V_4$ and also on the same plane as $V_2,V_3,V_4$.
Thus $V_3'$ lies on the same line as $V_3$ and $V_4$.
We conclude that the two copies form a tetrahedron $V_2V_3V_1V_3'$
with $V_4$ being the midpoint of the edge $V_3V_3'$.
This new tetrahedron has a sextuple of dihedral angles
$$\paren{\alpha,\pi-2\alpha,\frac{\pi}{2},\frac{\pi}{2},\frac{\pi}{3},\alpha},$$
which corresponds to Goldberg's first infinite family.
By Proposition \ref{prop:dihedralfixes}, it belongs to Goldberg's first infinite family.
Because RC6 is formed by slicing a member of Goldberg's first
infinite family in half along an edge with dihedral angle
$\pi/3$, RC6 belongs to Goldberg's second infinite family.
\end{proof}

\subsection*{Remaining case 7 (RC7)}
RC7 belongs to the code type \texttt{abccbb}. It is characterized by
\begin{equation}
\label{eq:RC7}
\text{RC7: } \begin{cases}
& \theta_{12}=\pi/3, \\
& \theta_{14}= \theta_{23}=\pi/2, \\
& \theta_{13} = \theta_{24}, \\
& 2\theta_{13} + \theta_{34} = \pi.
\end{cases} 
\end{equation}

\begin{prop}
\label{prop:RC7}
RC7 is a member of Goldberg's first infinite family.
\end{prop}

\begin{proof}
From (\ref{eq:RC7}), the dihedral angles of RC7 are $$(\theta_{12},\theta_{13},\theta_{14},\theta_{23},\theta_{24},\theta_{34})
= \paren{\frac{\pi}{3}, \alpha, \frac{\pi}{2}, \frac{\pi}{2}, \alpha, \pi-2\alpha}.$$
By Proposition \ref{prop:dihedralfixes}, RC7 belongs to the first Goldberg family.
\end{proof}

We conclude Section \ref{sec:10types} by proving that the Sommerville No. 1
has less surface area than any tile of the ten non-characterized types of Figure \ref{fig:tetrahedrontable}.

\begin{theorem}
\label{thm:10types_sono1best}
The Sommerville No. 1. has less surface area than
any tetrahedral tile of the ten non-characterized types of Figure \ref{fig:tetrahedrontable}.
\end{theorem}

\begin{proof}
By Proposition \ref{prop:allfromcode}, the Sommerville No. 1
has less surface area than any tile of the ten non-characterized types of Figure \ref{fig:tetrahedrontable}
except possibly for the remaining cases RC1--RC7.
By Proposition \ref{prop:RC1}, RC1 does not tile.
By Proposition \ref{prop:RC3}, RC3 has greater surface area than the Sommerville No. 1.
By Propositions \ref{prop:RC2}, \ref{prop:RC4}, \ref{prop:RC5}, \ref{prop:RC6} and \ref{prop:RC7}, the other remaining cases are among the tiles identified in previous literature; 
by Corollary \ref{cor:prevknown_so1best}, the Sommerville No. 1 minimizes surface area among these cases.
Therefore the Sommerville No. 1 has less surface area than
any tile of the ten non-characterized types of Figure \ref{fig:tetrahedrontable}.
\end{proof}

\section{The Best Tetrahedral Tile}
\label{sec:besttile}

Section \ref{sec:besttile} concludes the paper and provides some conjectures about tetrahedral tiles and $n$-hedral tiles in general. Extension to other $n$-hedral tiles and non-face-to-face tiles is briefly discussed. We begin with Theorem \ref{thm:bigtheorem}, which states that the Sommerville No. 1 is the best tetrahedral tile.

\begin{theorem}
\label{thm:bigtheorem}
The least-surface-area, face-to-face tetrahedral tile
of unit volume is uniquely the Sommerville No. 1.
\end{theorem}

\noindent (See Figure \ref{fig:sommerville1} of the Introduction.)

\begin{proof}
Proposition \ref{prop:25typescomplete} classifies all tetrahedra into the 25 types of Figure \ref{fig:tetrahedrontable}.
By Theorem \ref{thm:15types_sono1best}, the Sommerville No. 1 uniquely minimizes surface area among the tiles of the fifteen types in the orientation-preserving, Edmonds, and $2\pi/n$ groups circled in Figure \ref{fig:tetrahedrontable}. 
By Theorem \ref{thm:10types_sono1best}, the Sommerville No. 1 has less surface area than the tiles of the remaining ten types in Figure \ref{fig:tetrahedrontable}. 
Therefore the Sommerville No. 1 uniquely minimizes surface area among all tetrahedral tiles.
\end{proof}

Although this paper did not completely characterize all tetrahedral tiles, we conjecture that the list of tetrahedral tiles as outlined in Section \ref{sec:prevtiles} is complete.

\begin{conj}
The Sommerville and Goldberg lists of face-to-face tetrahedral tiles are complete. In particular, a tetrahedral tile has at least one dihedral angle of $\pi/2$.
\end{conj}

We also conjecture that any area-minimizing $n$-hedral tile is convex and is identical to its mirror image.

\begin{conj}
For any $n$, the least-area $n$-hedral tile is convex.
\end{conj}

\begin{conj}
For any $n$, the least-area $n$-hedral tile is congruent to its mirror image.
\end{conj}

\begin{remark}
This is consistent with the proven tiles for $n=4,5,6$. It is also consistent with the conjectured least-area tiles for $7 \leq n \leq 14$ presented by Gallagher \emph{et al.} \cite{Gal}.
\end{remark}

Finally, we comment on how the methods in this paper might be applicable to $n$-hedral tiles for $n>4$ and to non-face-to-face tilings.

\subsection*{$n$-hedra}
For $n>4$, not all $n$-hedra are combinatorially equivalent.
This may present more difficulty, but there are also potentially more tiling
restrictions because different polygonal shapes cannot coincide in a face-to-face tile.
An extra difficulty might also be that the dihedral angles no longer determine
the polyhedron as in the tetrahedron case (Prop. \ref{prop:dihedralfixes}).
One of the main tools developed in this paper is the edge-length graph
(Sect. \ref{sec:edgelengraphs}). For general $n$-hedra, we can easily generalize
the edge-length graph to have nodes being $(n-1)$-tuples.
However, this loses some information, as non-triangular faces cannot be uniquely identified by edge lengths alone.
Bounds on the dihedral angles in the spirit of Section \ref{sec:boundiso}
may also be obtained by observing that for a good isoperimetric candidate,
for a given volume, the diameter cannot be large. Then for convex polyhedra,
a dihedral angle cannot be very small, otherwise the volume would be close to 0.
Combinatorial explosion may also be a concern:
in the next unproven case, $n=7$, the combinatorial shape of the pentagonal prism
has 15 edges. Suppose each dihedral angle is of the form $2\pi/n$
in the spirit of Theorem \ref{thm:2pi/ntile}. If there are 20 values of $n$,
then the number of cases could be of the order $10^{19}$.
Similarly, the \numcasesinitial\ cases considered by code in Section \ref{sec:code}
may increase substantially in number.

\subsection*{Non-face-to-face tiles}
For non-face-to-face tetrahedral tiles, the linear equations have to be modified.
If two edges of different lengths coincide, then the linear combinations would
contain dihedral angles associated with different edge lengths,
unlike those presented in Table \ref{tab:10typesedgelengraphs}.
If an edge coincides with a face of an adjacent tile, then the dihedral angles
will sum to $\pi$ instead of $2\pi$. Both of these effects can also occur together.
The difficulty may be the classification of all tiling behaviors that arise from these effects.
Imposing the assumption that the tiling is periodic
may help in navigating these difficulties.



\bibliographystyle{alpha}

\vspace{.7cm}
\noindent Eliot Bongiovanni\\
Michigan State University \\
\url{eliotbonge@gmail.com}

\vspace{.7cm}
\noindent Alejandro Diaz\\
University of Maryland, College Park \\
\url{diaza5252@gmail.com}

\vspace{.7cm}
\noindent Arjun Kakkar\\
Williams College \\
\url{ak23@williams.edu}

\vspace{.7cm}
\noindent Nat Sothanaphan\\
Massachusetts Institute of Technology \\
\url{natsothanaphan@gmail.com}

\end{document}